\documentclass[11pt,reqno]{amsart}

\usepackage{amsmath, amsthm, amssymb,longtable,enumitem,multirow,hyperref}

\usepackage[a4paper, margin=1in]{geometry}

\title[GRADED ALGEBRAS AND SIGN CHANGES OF FOURIER COEFFICIENTS FOR ETA QUOTIENTS]
	{Graded Algebras of Modular Forms and Sign Changes
	of Fourier Coefficients for Eta Quotients}
\author{Jianwen Gan}
\address{School of Mathematical Sciences, Beijing Normal University, Beijing 100875, China}
\email{202631130037@mail.bnu.edu.cn}
\date{September 15, 2026}

\newtheorem{theorem}{Theorem}[section]      
\newtheorem{lemma}{Lemma}[section]      

\begin{document}

\begin{abstract}
	In this paper, 
	for $N=8,12,16,20,32$, 
	we determine the explicit structure of the graded algebra of modular forms 
	for $\varGamma_0(N)$ of both integral and half-integral weights, 
	with all quadratic Dirichlet characters modulo $N$. 
	We give explicit eta quotient generators 
	and determine the relations among them. 
	Using these structures, 
	for a given weight, 
	a level among those considered above, 
	and a quadratic character, 
	we obtain a basis for the corresponding space of modular forms 
	consisting of eta quotients. Using these eta quotient bases, 
	we obtain necessary and sufficient conditions for the periodicity 
	of the signs of the Fourier coefficients of modular forms 
	in three specific spaces. 
	Finally, 
	we determine all eta quotients in these spaces 
	whose Fourier coefficient signs are periodic.
\end{abstract}
\maketitle

\section{Introduction and statement of results}

For a congruence subgroup $\varGamma$ of $\mathrm{SL}_2(\mathbb{Z})$, 
we denote by $\mathcal{M}_k(\varGamma)$ 
the $\mathbb{C}$-vector space of modular forms of weight $k\in\mathbb{Z}$ 
on $\varGamma$. 
We know from \cite{BB1966} that the graded \(\mathbb{C}\)-algebra
\[
\mathcal{M}(\varGamma)=\bigoplus_{k\in \mathbb{Z}, k\geqslant 0}\mathcal{M}_k(\varGamma)
\]
is finitely generated.
For $\varGamma(N)$, with $N\geqslant 3$, Khuri-Makdisi proves in \cite{KM2012}
that modular forms of weight $1$ suffice to generate $\mathcal{M}(\varGamma(N))$.
Rustom proves in \cite{RUSTOM2014}
that $\mathcal{M}(\varGamma_1(N))$ is generated in weight at most $3$ for $N\geqslant 5$,
and that $\mathcal{M}(\varGamma_0(N))$ is generated in weight at most $6$ for $N$ satisfying certain conditions.

The Dedekind eta function is the infinite product
$$
\eta(z)=e^{2\pi iz/24}\prod_{n=1}^{\infty}(1-q^n),
$$
with
$q=e^{2\pi iz}$
for 
$z \in \mathcal{H}=\{z\in \mathbb{C} \mid \mathrm{Im}(z)>0\}$.
An eta quotient $f(z)$ is of the form
$$
f(z)=\prod_{\delta\mid N}\eta(\delta z)^{r_\delta},
$$
where $N\geqslant 1$ and $r_\delta \in \mathbb{Z}$.
There are precisely 121 positive integers $N\leqslant 500$
so that $\mathcal{M}(\varGamma_0(N))$
is generated by eta quotients, as proved by Rouse and Webb in \cite{RW2015}.

We denote the $\mathbb{C}$-vector space 
of weight $k \in \frac{1}{2}\mathbb{Z}$ modular (resp. cusp) forms 
on $\varGamma_0(N)$ (if $k\in \frac{1}{2}+\mathbb{Z}$, than $4 \mid N$) 
with Nebentypus character $\chi$ by $\mathcal{M}_k(\varGamma_0(N),\chi)$
(resp.$\mathcal{S}_k(\varGamma_0(N),\chi)$).
If $\chi = \chi_0$ is the trivial character modulo $N$,
then we use the notation
$\mathcal{M}_k(\varGamma_0(N))$ (resp. $\mathcal{S}_k(\varGamma_0(N))$).
We denote the direct sum over all quadratic characters by
\begin{align*}
    \mathcal{M}_k^{\mathrm{quad}}(\varGamma_0(N))=\bigoplus_{\chi \, \mathrm{quadratic}}\mathcal{M}_k(\varGamma_0(N),\chi),
\end{align*}
\begin{align*}
	\mathcal{S}_k^{\mathrm{quad}}(\varGamma_0(N)) = \bigoplus_{\chi \, \mathrm{quadratic}} \mathcal{S}_k(\varGamma_0(N),\chi),
\end{align*}
where the superscript indicates restriction to quadratic characters 
(i.e., $\chi$ takes values {$\pm 1$}).
Then we have the graded $\mathbb{C}$-algebra
\[
\mathcal{M}_{\mathbb{Z}/2}^{\mathrm{quad}}(\varGamma_0(N))=\bigoplus_{k\in \frac{1}{2}\mathbb{Z}, k\geqslant 0}\mathcal{M}_k^{\mathrm{quad}}(\varGamma_0(N)),
\]
\[
\mathcal{S}_{\mathbb{Z}/2}^{\mathrm{quad}}(\varGamma_0(N))=\bigoplus_{k\in \frac{1}{2}\mathbb{Z}, k\geqslant 0}\mathcal{S}_k^{\mathrm{quad}}(\varGamma_0(N)).
\]
$\mathcal{S}_{\mathbb{Z}/2}^{\mathrm{quad}}(\varGamma_0(N)) $ is a homogeneous ideal of $\mathcal{M}_{\mathbb{Z}/2}^{\mathrm{quad}}(\varGamma_0(N))$.
It is a classical result that 
\begin{equation}
	\mathcal{M}_{\mathbb{Z}/2}^{\mathrm{quad}}(\varGamma_0(4))=\mathbb{C}[\frac{\eta(2z)^5}{\eta(z)^2\eta(4z)^2},\frac{\eta(4z)^8}{\eta(2z)^4}]
	\label{N=4}
\end{equation}
(see \cite{Cohen1975, Koblitz1993}).
This is an example where 
the graded algebra of modular forms for $\varGamma_0(N)$ 
of both integral and half-integral weights with quadratic Dirichlet character
can be generated by eta quotients.
Motivated by this identity, 
we seek analogous generating sets of eta quotients for other levels $N$. 

A key outcome of our work is the determination of the explicit structure 
(as a graded algebra) for these spaces of modular forms.
From \eqref{N=4}, we obtain the following simple description of the graded algebra
$$
\mathcal{M}_{\mathbb{Z}/2}^{\mathrm{quad}}(\varGamma_0(4))
\cong \mathbb{C}[x_1,x_2], 
\quad \frac{\eta(2z)^5}{\eta(z)^2\eta(4z)^2} \mapsto x_1,
\quad \frac{\eta(4z)^8}{\eta(2z)^4} \mapsto x_2,
$$
with $\deg(x_1)=1$ and $\deg(x_2)=4$. 
This grading follows the convention 
that a generator's degree is twice the weight of the corresponding modular form, 
hence $x_1$ and $x_2$ correspond to modular forms of weight $\frac{1}{2}$ and $2$, respectively, 
yielding a standard integer grading.
For a general level $N$, 
the graded algebra typically takes the form
$$
\mathcal{M}_{\mathbb{Z}/2}^{\mathrm{quad}}(\varGamma_0(N))\cong \mathbb{C}[x_1,x_2,\cdots x_r]/I,
$$
where $\mathbb{C}[x_1,x_2,\cdots x_r]$ is a graded polynomial ring 
(each $x_i$ assigned a degree) 
and $I$ is a homogeneous ideal encoding the relations among the corresponding modular forms.
In this paper, we abuse notation and write $x_i$ 
for the class $x_i+I$ in the quotient ring.

Using the isomorphism described above, 
we obtain an explicit basis of $\mathcal{M}_k^{\mathrm{quad}}(\varGamma_0(N))$. 
In this paper, 
the basis elements are specifically chosen to be eta quotients. 
The Fourier coefficients of eta quotients 
are not only relatively straightforward to compute 
but also frequently encode interesting combinatorial information. 
This makes our chosen basis particularly effective 
for studying various arithmetic properties of 
the Fourier coefficients of modular forms in this space, 
such as the periodicity of their sign changes.

The Fourier expansion of an eta quotient can be uniformly written as
$$
f(z)=\prod_{\delta\mid N}\eta(\delta z)^{r_\delta}=
q^{a_f}\sum_{n=0}^{\infty}C_{1^{r_1}\cdots N^{r_N}}(n)q^n,
\quad a_f=\frac{1}{24}\sum_{\delta \mid N}\delta r_\delta,
$$
which yields a systematic notation $C_{1^{r_1}\cdots N^{r_N}}(n)$ for its Fourier coefficients.
Some of the coefficients $C_{1^{r_1}\cdots N^{r_N}}(n)$ correspond to classical arithmetic functions. 
For instance,
$C_{1^{-1}}(n)=p(n)$
is the ordinary partition function, and
$C_{1^{-6}2^{15} 4^{-6}}(n)=r_3(n)$
gives the number of representations of $n$ as a sum of three squares.

Bringmann, Han, Heim, and Kane in \cite{BHHK2026}
explicitly compute the periodicity of the Fourier coefficients
of some holomorphic eta quotients, for example,
$
C_{1^{-2}2^{3}4^{2}}(n)>0
$
and
\[
\operatorname{sgn}\bigl(C_{1^{4}2^{2}4^{-2}}(n)\bigr)
=
\begin{cases}
1 & \text{if } n \equiv 0,3,7 \pmod{8},\\
-1 & \text{if } n \equiv 1,4,5 \pmod{8},\\
0 & \text{if } n \equiv 2 \pmod{4}.
\end{cases}
\]
In \cite{BHHK2026}, 
the signs of the Fourier coefficients of eta quotients
are computed individually.
In contrast, here we use eta quotients to construct a basis,
thereby obtaining necessary and sufficient conditions
for the periodicity of modular forms 
of a given weight, level, and quadratic character.
For a single space, one can determine the periodicity
of several dozen eta quotients.
For instance, \(C_{1^{-2}2^{3}4^{2}}(n)\) appears in
Table~\ref{tab:eta_periodic_32_2}, and
\(C_{1^{4}2^{2}4^{-2}}(n)\) appears in
Table~\ref{tab:eta_periodic_16}.

We compute the structure of $\mathcal{M}_{\mathbb{Z}/2}^{\mathrm{quad}}(\varGamma_0(N))$ for $N = 8, 12, 16, 20, 32$.
\begin{theorem}\label{tho:1}
	For $N=8$,
	\begin{align*}
		\mathcal{M}_{\mathbb{Z}/2}^{\mathrm{quad}}(\varGamma_0(8))
		=\mathbb{C}\!\left[\frac{\eta(2z)^5}{\eta(z)^2\eta(4z)^2},
		\frac{\eta(4z)^5}{\eta(2z)^2\eta(8z)^2}\right]
    	\cong \mathbb{C}[x_1,x_2],
	\end{align*}
	\begin{align*}
		\mathcal{S}_{\mathbb{Z}/2}^{\mathrm{quad}}(\varGamma_0(8))
		=\eta(z)^2\eta(2z)\eta(4z)\eta(8z)^2\cdot
		\mathbb{C}\!\left[\frac{\eta(2z)^5}{\eta(z)^2\eta(4z)^2},
		\frac{\eta(4z)^5}{\eta(2z)^2\eta(8z)^2}\right],
	\end{align*}
	where
	$\displaystyle
	\frac{\eta(2z)^5}{\eta(z)^2\eta(4z)^2} \mapsto x_1,
	\frac{\eta(4z)^5}{\eta(2z)^2\eta(8z)^2} \mapsto x_2,
	$
	with $\deg(x_1)=\deg(x_2)=1$.
	The eta quotient $\eta(z)^2 \eta(2z) \eta(4z) \eta(8z)^2$ 
	corresponds to $\displaystyle \frac{1}{4}(-2x_1x_2^5 + 3x_1^3x_2^3 - x_1^5x_2)$. 
\end{theorem}
\begin{theorem}\label{tho:2}
	For $N = 12$,
	\begin{align*}
	\mathcal{M}_{\mathbb{Z}/2}^{\mathrm{quad}}(\varGamma_0(12))
	&= \mathbb{C}\!\left[
	\frac{\eta(2z)^5}{\eta(z)^2\eta(4z)^2},\,
	\frac{\eta(6z)^5}{\eta(3z)^2\eta(12z)^2},\,
	\frac{\eta(4z)^2\eta(12z)^2}{\eta(2z)\eta(6z)},\,
	\frac{\eta(2z)\eta(12z)^6}{\eta(4z)^2\eta(6z)^3}
	\right] \\
	&\cong \mathbb{C}[x_1,x_2,x_3,x_4]/I_{12},
	\end{align*}
	\begin{align*}
		\mathcal{S}_{\mathbb{Z}/2}^{\mathrm{quad}}(\varGamma_0(12))
		\cong J_{12}/I_{12},
	\end{align*}
	where the ideal $I_{12}$ is generated by
	\begin{align*}
	&x_1^3 - x_1 x_2^2 - 4x_2x_3 - 12x_2x_4, \;
	x_1^2x_2 - x_2^3 - 4x_1x_3 + 4x_1x_4, \\
	&x_1^4 - 2x_1^2x_2^2 + x_2^4 - 16x_3^2 - 32x_3x_4 + 48x_4^2,
	\end{align*}
	and the ideal $J_{12}$ is generated by $I_{12}$ and 
	\begin{align*}
	x_2^3x_3-4x_1x_3x_4,\;x_1x_2^2x_3-4x_1x_2^2x_4+8x_2x_3x_4-12x_2x_4^2.
	\end{align*}
	The isomorphism is given by
	\begin{align*}
	\frac{\eta(2z)^5}{\eta(z)^2\eta(4z)^2} \mapsto x_1, \;
	\frac{\eta(6z)^5}{\eta(3z)^2\eta(12z)^2} \mapsto x_2, \;
	\frac{\eta(4z)^2\eta(12z)^2}{\eta(2z)\eta(6z)} \mapsto x_3, \;
	\frac{\eta(2z)\eta(12z)^6}{\eta(4z)^2\eta(6z)^3} \mapsto x_4,
	\end{align*}
	with $\deg(x_1) = \deg(x_2) = 1$ and $\deg(x_3) = \deg(x_4) = 2$.
	Using the isomorphism,
	\begin{align*}
	x_2^3x_3-4x_1x_3x_4 
	&\mapsto \frac{\eta(2z)^3\eta(3z)^2\eta(12z)^2}{\eta(6z)^2},\\
	x_1x_2^2x_3-4x_1x_2^2x_4+8x_2x_3x_4-12x_2x_4^2 
	&\mapsto \frac{\eta(z)^2\eta(4z)^2\eta(6z)^3}{\eta(2z)^2}.
	\end{align*}
\end{theorem}
\begin{theorem}\label{tho:3}
	For $N=16$,
	\begin{align*}
		\mathcal{M}_{\mathbb{Z}/2}^{\mathrm{quad}}(\varGamma_0(16))
		&=\mathbb{C}\!\left[\frac{\eta(2z)^5}{\eta(z)^2\eta(4z)^2},
		\frac{\eta(4z)^5}{\eta(2z)^2\eta(8z)^2},\frac{\eta(16z)^2}{\eta(8z)}\right]\\
    	&\cong \mathbb{C}[x_1,x_2,x_3]/\langle x_1^2-4x_1x_3-x_2^2+8x_3^2\rangle,
	\end{align*}
	\begin{align*}
	\mathcal{S}_{\mathbb{Z}/2}^{\mathrm{quad}}(\varGamma_0(16))
	=\eta(2z)^2\eta(4z)\eta(8z)^2\cdot 
	\mathbb{C}\!\left[\frac{\eta(2z)^5}{\eta(z)^2\eta(4z)^2},
	\frac{\eta(4z)^5}{\eta(2z)^2\eta(8z)^2},\frac{\eta(16z)^2}{\eta(8z)}\right],
	\end{align*}
	where
	$\displaystyle
	\frac{\eta(2z)^5}{\eta(z)^2\eta(4z)^2} \mapsto x_1,
	\frac{\eta(4z)^5}{\eta(2z)^2\eta(8z)^2} \mapsto x_2,
	\frac{\eta(16z)^2}{\eta(8z)} \mapsto x_3,$
	with $\deg(x_i)=1$. 
	The eta quotient $\eta(2z)^2\eta(4z)\eta(8z)^2$ 
	corresponds to $\displaystyle x_1x_2x_3(x_1-2x_3)(x_1-4x_3)$.
\end{theorem}
\begin{theorem}\label{tho:4}
For $N = 20$,
\begin{align*}
\mathcal{M}_{\mathbb{Z}/2}^{\mathrm{quad}}(\varGamma_0(20))
\cong \mathbb{C}[x_1,x_2,\cdots,x_7]/I_{20},
\end{align*}
where the ideal $I_{20}$ is generated by the following $15$ polynomials
{\small
$$
\begin{array}{ll}
\multicolumn{2}{l}{x_{1}^{4} - 6 x_{1}^{2} x_{2}^{2} + 5 x_{2}^{4} + 16 x_{1} x_{2} x_{3} - 16 x_{3}^{2},}\\
x_{1}^{2} x_{2} x_{3} - x_{2}^{3} x_{3} - 4 x_{1} x_{3}^{2} - 16 x_{1} x_{4} + 64 x_{1} x_{5}, &x_{1} x_{4} - 4 x_{1} x_{5} - x_{2} x_{6} + 4 x_{2} x_{7},\\ 
x_{1}^{4} x_{2} - x_{2}^{5} - 6 x_{1}^{3} x_{3} - 2 x_{1} x_{2}^{2} x_{3} + 8 x_{2} x_{3}^{2} - 96 x_{2} x_{4}, &x_{2} x_{4} - x_{1} x_{6},\\
x_{2}^{2} x_{5} - x_{1} x_{2} x_{7} - x_{3} x_{6}, & x_{1}^{2} x_{4} - x_{2}^{2} x_{4} - 3 x_{1}^{2} x_{5} - x_{2}^{2} x_{5} + 4 x_{3} x_{7},\\ 
x_{1} x_{2} x_{4} - 3 x_{1} x_{2} x_{5} - x_{2}^{2} x_{6} - x_{2}^{2} x_{7} - x_{3} x_{4} + 4 x_{3} x_{5}, & x_{1} x_{2} x_{5} - x_{1}^{2} x_{7} - x_{3} x_{4},\\
\multicolumn{2}{l}{x_{5}^{2} - x_{6} x_{7} - x_{7}^{2},}\\ 
\multicolumn{2}{l}{4 x_{1}^{2} x_{2}^{2} x_{5} - x_{1}^{3} x_{2} x_{6} + x_{1} x_{2}^{3} x_{6} - 3 x_{1} x_{2}^{3} x_{7} - x_{1} x_{2} x_{3} x_{5} - x_{3}^{2} x_{5} - x_{4}^{2} + 16 x_{5}^{2} - 4 x_{6}^{2} - 4 x_{6} x_{7},}\\ 
\multicolumn{2}{l}{7 x_{1}^{2} x_{2}^{2} x_{5} + x_{2}^{4} x_{5} - 2 x_{1}^{3} x_{2} x_{6} + 2 x_{1} x_{2}^{3} x_{6} - 3 x_{1} x_{2}^{3} x_{7} - x_{1} x_{2} x_{3} x_{5} - 3 x_{2}^{2} x_{3} x_{7} - 2 x_{3}^{2} x_{5} - x_{4}^{2} + 16 x_{5}^{2} - 5 x_{6}^{2},}\\ 
\multicolumn{2}{l}{3 x_{1}^{2} x_{2}^{2} x_{5} + x_{2}^{4} x_{5} - x_{1}^{3} x_{2} x_{6} + x_{1} x_{2}^{3} x_{6} - x_{1}^{2} x_{3} x_{7} - 3 x_{3}^{2} x_{5} + x_{4}^{2} - 4 x_{4} x_{5},}\\ 
\multicolumn{2}{l}{3 x_{1}^{3} x_{2} x_{5} + 5 x_{1} x_{2}^{3} x_{5} - x_{1}^{4} x_{6} + x_{2}^{4} x_{6} - x_{1}^{2} x_{2}^{2} x_{7} - 3 x_{2}^{2} x_{3} x_{5} - 3 x_{1} x_{2} x_{3} x_{7} + 4 x_{5} x_{6} + 4 x_{4} x_{7} - 16 x_{5} x_{7},}\\ 
\multicolumn{2}{l}{3 x_{1}^{3} x_{2} x_{5} + 5 x_{1} x_{2}^{3} x_{5} - x_{1}^{4} x_{6} + x_{2}^{4} x_{6} - 3 x_{2}^{4} x_{7} - x_{1}^{2} x_{3} x_{5} - x_{1} x_{2} x_{3} x_{7} - 3 x_{3}^{2} x_{7} - x_{4} x_{6}.}
\end{array}
$$
}
The isomorphism is given by
\begin{gather*}
	\frac{\eta(2z)^5}{\eta(z)^2\eta(4z)^2} \mapsto x_1, \;
	\frac{\eta(10z)^5}{\eta(5z)^2\eta(20z)^2} \mapsto x_2, \;
	\frac{\eta(z)\eta(2z)\eta(10z)\eta(20z)}{\eta(4z)\eta(5z)} \mapsto x_3,
	\allowdisplaybreaks \\
	\frac{\eta(2z)^5\eta(5z)^2\eta(20z)^7}{\eta(z)^2\eta(4z)^3\eta(10z)^5} \mapsto x_4,\;
	\frac{\eta(20z)^8}{\eta(10z)^4} \mapsto x_5, \;
	\frac{\eta(20z)^5}{\eta(4z)} \mapsto x_6, \;
	\frac{\eta(2z)\eta(20z)^{10}}{\eta(4z)^2\eta(10z)^5} \mapsto x_7,
\end{gather*}
with $\deg(x_1) = \deg(x_2) = 1$, $\deg(x_3)=2$ and $\deg(x_4)=\deg(x_5)=\deg(x_6)=\deg(x_7)=4$.
\end{theorem}
\begin{theorem}\label{tho:5}
	For $N=32$,
	\begin{align*}
		\mathcal{M}_{\mathbb{Z}/2}^{\mathrm{quad}}(\varGamma_0(32))
		&=\mathbb{C}\!\left[\frac{\eta(2z)^5}{\eta(z)^2\eta(4z)^2},
		\frac{\eta(4z)^5}{\eta(2z)^2\eta(8z)^2},\frac{\eta(16z)^2}{\eta(8z)},
		\frac{\eta(32z)^2}{\eta(16z)}\right]\\
    	&\cong \mathbb{C}[x_1,x_2,x_3,x_4]/\langle x_1^2-x_2^2+8x_3^2-4x_1x_3,x_3^2+2x_4^2-x_2x_4\rangle,
	\end{align*}
	\begin{align*}
	\mathcal{S}_{\mathbb{Z}/2}^{\mathrm{quad}}(\varGamma_0(32))
	=\eta(4z)^2\eta(8z)^2\cdot 
	\mathbb{C}\!\left[\frac{\eta(2z)^5}{\eta(z)^2\eta(4z)^2},
	\frac{\eta(4z)^5}{\eta(2z)^2\eta(8z)^2},\frac{\eta(16z)^2}{\eta(8z)},
	\frac{\eta(32z)^2}{\eta(16z)}\right],
	\end{align*}
	where
	$\displaystyle
	\frac{\eta(2z)^5}{\eta(z)^2\eta(4z)^2} \mapsto x_1,
	\frac{\eta(4z)^5}{\eta(2z)^2\eta(8z)^2} \mapsto x_2,
	\frac{\eta(16z)^2}{\eta(8z)} \mapsto x_3,
	\frac{\eta(32z)^2}{\eta(16z)} \mapsto x_4,$
	with $\deg(x_i)=1$. 
	The eta quotient $\eta(4z)^2\eta(8z)^2$ 
	corresponds to $\displaystyle x_3(x_1-2x_3)(x_1^2+8x_3^2-4x_1x_3-4x_2x_4)$.
\end{theorem}

From the above theorems, for $N = 8, 16, 32$, 
we can determine whether the signs 
of the Fourier coefficients of eta quotients 
in $\mathcal{M}_{3/2}^{\mathrm{quad}}(\varGamma_0(N))$
and $\mathcal{M}_{2}^{\mathrm{quad}}(\varGamma_0(N))$ exhibit periodicity.
Since we study the signs of Fourier coefficients, 
all modular forms appearing hereafter are assumed to have real coefficients.
In the following, 
we exclude duplicate eta quotients, 
i.e., those satisfying $f(z) = f(mz)$.
Throughout, 
we let $\left(\frac{a}{n}\right)$ be the Kronecker symbol.
Given an integer $t$, 
let $\chi_t$ denote the corresponding Dirichlet character. 
If $t$ is a  square, then $\chi_t = \mathbf{1}$. 
If $t$ is not a square, 
let $D$ be the discriminant of the quadratic field $\mathbb{Q}(\sqrt{t})$ 
and define $\chi_t(m) = \left(\frac{D}{m}\right)$.

\begin{theorem}\label{tho:6}
	If $f(z) \in \mathcal{M}_{3/2}(\varGamma_0(2^r))$ 
	with $3 \leqslant r \leqslant 5$, 
	write $f(z) = \sum_{n=0}^{\infty} a(n) q^n$.
	When $a(0) = 0$, the sign $\operatorname{sgn}(a(n))$ is periodic. 
	Let $T$ denote its period, then $ T \mid 2^{r-1}$. 
	When $a(0) \neq 0$, the sign $\operatorname{sgn}(a(n))$ is not periodic.
	From these we know that $23$ eta quotients exhibit periodicity, 
	as listed in Table~\ref{tab:eta_periodic_32_1}.
	{\renewcommand{\arraystretch}{1.4}
	\begin{longtable}{|c|c|c|c|}
	\caption{Eta quotients with periodic sign patterns} 
	\label{tab:eta_periodic_32_1}\\[-1ex]
	\hline
	Period/sign pattern & Eta quotient & Period/sign pattern & Eta quotient\\
	\hline
	\endfirsthead  
	\multicolumn{4}{l}{continued from previous page} \\
	\hline
	Period/sign pattern & Eta quotient & Period/sign pattern & Eta quotient\\
	\hline
	\endhead
	\hline
	\multicolumn{4}{r}{{continued on next page}} \\
	\endfoot
	\hline
	\endlastfoot
	$1/$\texttt{+} & $1^{-3}2^{6}$ &
	$1/$\texttt{+} & $1^{-4}2^{9}4^{-2}$ \\ \hline
	$1/$\texttt{+} & $1^{-1} 4^{4}$ &
	$1/$\texttt{+} & $1^{-5} 2^{12} 4^{-4}$ \\ \hline
	$2/$\texttt{+-} & $1^{4}2^{-3}4^{2}$ &
	$4/$\texttt{+++0} & $1^{-4}2^{10}4^{-5}8^{2}$ \\ \hline
	$4/$\texttt{++00} & $1^{-2}2^5 4^{-4}8^4$  &
	$4/$\texttt{++00} & $1^{-2} 2^{5}4^{-3}8^{1}16^2$ \\ \hline
	$4/$\texttt{++00} & $1^{-2} 2^{5}4^{-5}8^{7}16^{-2}$ &
	$4/$\texttt{+-+0} & $1^{4}2^{-2}4^{-1}8^{2}$ \\ \hline
	$4/$\texttt{+-00} & $1^{2}2^{-1}4^{-2}8^{4}$ &
	$4/$\texttt{+-00} & $1^{2} 2^{-1}4^{-1}8^{1}16^2$ \\ \hline
	$4/$\texttt{+-00} & $1^{2} 2^{-1}4^{-3}8^{7}16^{-2}$ &
	$8/$\texttt{+++0++00} & $1^{-4}2^{10}4^{-4}8^{-1}16^{2}$ \\ \hline
	$8/$\texttt{++00+000} & $1^{-2}2^{5}4^{-2}8^{-2}16^{4}$ &
	$8/$\texttt{+-+0+-00} & $1^{4}2^{-2}8^{-1}16^{2}$ \\ \hline
	$8/$\texttt{+-00+000} & $1^{2}2^{-1}8^{-2}16^{4}$ &
	$16/$\texttt{+++++0+0++0++000} & $ 1^{-2} 2^{3} 4^{3} 8^{-2} 16^{-1} 32^{2} $ \\ \hline
	$16/$\texttt{++00+0000+000000} & $ 1^{-2} 2^{5} 4^{-2} 16^{-2} 32^{4} $ &
	$16/$\texttt{++--+0-0++0-+000} & $ 1^{-2} 2^{7} 4^{-3} 16^{-1} 32^{2} $ \\ \hline
	$16/$\texttt{+-+-+0+0+-0-+000} & $ 1^{2} 2^{-3} 4^{5} 8^{-2} 16^{-1} 32^{2} $ &
	$16/$\texttt{+-00+0000-000000} & $ 1^{2} 2^{-1} 16^{-2} 32^{4} $ \\ \hline
	$16/$\texttt{+--++0-0+-0++000} & $ 1^{2} 2^{1} 4^{-1} 16^{-1} 32^{2} $ & &
	\end{longtable}}
\end{theorem}

\begin{theorem}\label{tho:7}
	If $f(z) \in \mathcal{M}_{3/2}(\varGamma_0(2^r),\chi_8 )$ 
	with $3 \leqslant r \leqslant 5$, 
	write $f(z) = \sum_{n=0}^{\infty} a(n) q^n$.
	When $a(0) = 0$, the sign $\operatorname{sgn}(a(n))$ is periodic.
	Let $T$ denote its period, then $ T \mid 2^{r-1}$. 
	When $a(0) \neq 0$, the sign $\operatorname{sgn}(a(n))$ is not periodic.
	From these we know that $16$ eta quotients exhibit periodicity, 
	as listed in Table~\ref{tab:eta_periodic_32_2}.
	{\renewcommand{\arraystretch}{1.4}
	\begin{longtable}{|c|c|c|c|}
	\caption{Eta quotients with periodic sign patterns} 
	\label{tab:eta_periodic_32_2}\\[-1ex]
	\hline
	Period/sign pattern & Eta quotient & Period/sign pattern & Eta quotient\\
	\hline
	\endfirsthead  
	\multicolumn{4}{l}{continued from previous page} \\
	\hline
	Period/sign pattern & Eta quotient & Period/sign pattern & Eta quotient\\
	\hline
	\endhead
	\hline
	\multicolumn{4}{r}{{continued on next page}} \\
	\endfoot
	\hline
	\endlastfoot
	$1/$\texttt{+} & $1^{-2}2^3 4^2$ &
	$1/$\texttt{+} & $ 1^{-2} 2^{4} 4^{-1} 8^{2} $ \\ \hline
	$1/$\texttt{+} & $ 1^{-2} 2^{2} 4^{5} 8^{-2} $ &
	$2/$\texttt{+-} & $ 1^{2} 2^{-3} 4^{4} $ \\ \hline
	$8/$\texttt{+++++0+0} & $1^{-2}2^{3}4^{3}8^{-3}16^{2}$ &
	$8/$\texttt{++00+000} & $ 1^{-2} 2^{5} 4^{-2} 8^{-1} 16^{1} 32^{2} $ \\ \hline
	$8/$\texttt{++00+000} & $ 1^{-2} 2^{5} 4^{-2} 8^{-3} 16^{7} 32^{-2} $ &
	$8/$\texttt{++--+0-0} & $ 1^{-2} 2^{7} 4^{-3} 8^{-1} 16^{2} $ \\ \hline
	$8/$\texttt{+-+-+0+0} & $1^{2}2^{-3}4^{5}8^{-3}16^{2}$ &
	$8/$\texttt{+-00+000} & $ 1^{2} 2^{-1} 8^{-1} 16^{1} 32^{2} $ \\ \hline
	$8/$\texttt{+-00+000} & $ 1^{2} 2^{-1} 8^{-3} 16^{7} 32^{-2} $ &
	$8/$\texttt{+--++0-0} & $ 1^{2} 2^{1} 4^{-1} 8^{-1} 16^{2} $ \\ \hline
	$16/$\texttt{+++0++00+++00+00} & $ 1^{-4} 2^{10} 4^{-4} 16^{-1} 32^{2} $ &
	$16/$\texttt{++00++00++000+00} & $ 1^{-2} 2^{5} 4^{-4} 8^{5} 16^{-3} 32^{2} $ \\ \hline
	$16/$\texttt{+-+0+-00+-+00-00} & $ 1^{4} 2^{-2} 16^{-1} 32^{2} $ &
	$16/$\texttt{+-00+-00+-000-00} & $ 1^{2} 2^{-1} 4^{-2} 8^{5} 16^{-3} 32^{2} $
	\end{longtable}}
\end{theorem}

\begin{theorem}\label{tho:8}
	If $f(z) \in \mathcal{M}_{2}(\varGamma_0(2^r))$
	with $2 \leqslant r \leqslant 4$,
	write $f(z) = \sum_{n=0}^{\infty} a(n) q^n$.
	The sign $\operatorname{sgn}(a(n))$ is periodic.
	Let $T$ denote its period, then $T \mid 2^{r-1}$. 
	From these we know that $36$ eta quotients exhibit periodicity, 
	as listed in Table~\ref{tab:eta_periodic_16}.
	{\renewcommand{\arraystretch}{1.4}
	\begin{longtable}{|c|c|c|c|}
	\caption{Eta quotients with periodic sign patterns} 
	\label{tab:eta_periodic_16}\\[-1ex]
	\hline
	Period/sign pattern & Eta quotient & Period/sign pattern & Eta quotient\\
	\hline
	\endfirsthead  
	\multicolumn{4}{l}{continued from previous page} \\
	\hline
	Period/sign pattern & Eta quotient & Period/sign pattern & Eta quotient\\
	\hline
	\endhead
	\hline
	\multicolumn{4}{r}{{continued on next page}} \\
	\endfoot
	\hline
	\endlastfoot
	$1/$\texttt{+} & $ 1^{-4} 2^{8} $ &
	$1/$\texttt{+} & $ 1^{-8} 2^{20} 4^{-8} $ \\ \hline
	$1/$\texttt{+} & $ 1^{-4} 2^{6} 4^{6} 8^{-4} $ &
	$1/$\texttt{+} & $ 1^{-2} 2^{2} 4^{4} $ \\ \hline
	$1/$\texttt{+} & $ 1^{-6} 2^{14} 4^{-4} $ &
	$1/$\texttt{+} & $ 1^{-2} 2^{1} 4^{6} 8^{1} 16^{-2} $ \\ \hline
	$1/$\texttt{+} & $ 1^{-6} 2^{15} 4^{-8} 8^{5} 16^{-2} $ &
	$2/$\texttt{+-} & $ 1^{8} 2^{-4} $ \\ \hline
	$2/$\texttt{+-} & $ 1^{4} 2^{-6} 4^{10} 8^{-4} $ &
	$2/$\texttt{+-} & $ 1^{4} 2^{-4} 4^{4} $ \\ \hline
	$2/$\texttt{+-} & $ 1^{2} 2^{-5} 4^{8} 8^{1} 16^{-2} $ &
	$2/$\texttt{+-} & $ 1^{6} 2^{-3} 4^{-2} 8^{5} 16^{-2} $ \\ \hline
	$4/$\texttt{+++0} & $ 1^{-4} 2^{10} 4^{-6} 8^{4} $ &
	$4/$\texttt{+++0} & $ 1^{-4} 2^{10} 4^{-8} 8^{10} 16^{-4} $ \\ \hline
	$4/$\texttt{++00} & $ 1^{-2} 2^{5} 4^{-4} 8^{3} 16^{2} $ &
	$4/$\texttt{++00} & $ 1^{-2} 2^{5} 4^{-6} 8^{9} 16^{-2} $ \\ \hline
	$4/$\texttt{++00} & $ 1^{-2} 2^{5} 4^{-8} 8^{15} 16^{-6} $ &
	$4/$\texttt{++--} & $ 1^{-2} 2^{9} 4^{-6} 8^{5} 16^{-2} $ \\ \hline
	$4/$\texttt{+-+0} & $ 1^{4} 2^{-2} 4^{-2} 8^{4} $ &
	$4/$\texttt{+-+0} & $ 1^{4} 2^{-2} 4^{-4} 8^{10} 16^{-4} $ \\ \hline
	$4/$\texttt{+-00} & $ 1^{2} 2^{-1} 4^{-2} 8^{3} 16^{2} $ &
	$4/$\texttt{+-00} & $ 1^{2} 2^{-1} 4^{-4} 8^{9} 16^{-2} $ \\ \hline
	$4/$\texttt{+-00} & $ 1^{2} 2^{-1} 4^{-6} 8^{15} 16^{-6} $ &
	$4/$\texttt{+--+} & $ 1^{2} 2^{3} 4^{-4} 8^{5} 16^{-2} $ \\ \hline
	$8/$\texttt{+++++++0} & $ 1^{-2} 2^{1} 4^{8} 8^{-5} 16^{2} $ &
	$8/$\texttt{+++++++0} & $ 1^{-6} 2^{15} 4^{-6} 8^{-1} 16^{2} $ \\ \hline
	$8/$\texttt{+++0++00} & $ 1^{-4} 2^{10} 4^{-4} 8^{-2} 16^{4} $ &
	$8/$\texttt{++00+000} & $ 1^{-2} 2^{5} 4^{-2} 8^{-3} 16^{6} $ \\ \hline
	$8/$\texttt{++0--+0-} & $ 1^{-4} 2^{14} 4^{-6} $ &
	$8/$\texttt{++--++-0} & $ 1^{-2} 2^{9} 4^{-4} 8^{-1} 16^{2} $ \\ \hline
	$8/$\texttt{+-+0+-00} & $ 1^{4} 2^{-2} 8^{-2} 16^{4} $ &
	$8/$\texttt{+-+-+-+0} & $ 1^{6} 2^{-3} 8^{-1} 16^{2} $ \\ \hline
	$8/$\texttt{+-+-+-+0} & $ 1^{2} 2^{-5} 4^{10} 8^{-5} 16^{2} $ &
	$8/$\texttt{+-0+--0+} & $ 1^{4} 2^{2} 4^{-2} $ \\ \hline
	$8/$\texttt{+-00+000} & $ 1^{2} 2^{-1} 8^{-3} 16^{6} $ &
	$8/$\texttt{+--++--0} & $ 1^{2} 2^{3} 4^{-2} 8^{-1} 16^{2} $
	\end{longtable}}
\end{theorem}

\section{Preliminaries}

\subsection{Modular forms} 
We briefly introduce the definitions of modular forms 
of integer and half-integer weight,  
for more details, see \cite{Koblitz1993,Ono2004,Diamond2005}.
For odd $d$,
define $\varepsilon_d$ by 
$$
	\varepsilon_d=\begin{cases}
	1 & \text{if\;} d \equiv 1 \pmod 4,\\
	i & \text{if\;} d \equiv 3 \pmod 4.
\end{cases}
$$
Throughout, we let $\sqrt{z}$ be the branch of the square root having argument in
$(-\pi/2,\pi/2]$.
If $f(z)$ is a meromorphic function on $\mathcal{H}$ and $\varGamma=\begin{pmatrix}
	a & b\\ c & d
\end{pmatrix} \in \text{SL}_2(\mathbb{Z})$, then define
the “slash” operator $\mid_k$ by
$$
(f\mid_k\varGamma)(z)=\begin{cases}
	(cz+d)^{-k}f(\varGamma z) & \text{if\;} k\in \mathbb{Z}, \varGamma \in \text{SL}_2(\mathbb{Z}),\\
	\left(\frac{c}{d}\right)^{-2k}\varepsilon_d^{2k} (cz+d)^{-k}f(\varGamma z) & \text{if\;} k\in \mathbb{Z}+\frac{1}{2}, \varGamma \in \varGamma_0(4).
\end{cases}
$$
A holomorphic function $f(z)$ on $\mathcal{H}$ is called a modular (resp. cusp) form 
of weight $k \in \frac{1}{2}\mathbb{Z}$ on $\varGamma_0(N)$ (if $k\in \frac{1}{2}+\mathbb{Z}$, than $4 \mid N$) 
with Nebentypus $\chi$ 
if $f\mid_k\varGamma=\chi(d)f$ for all 
$\varGamma=\begin{pmatrix}
	a & b\\ c & d
\end{pmatrix} \in \varGamma_0(N)$
and $\lim_{z\rightarrow i\infty}(cz+d)^{-k}f(\varGamma z)$ exists (resp. vanishes) for all 
$\varGamma=\begin{pmatrix}
	a & b\\ c & d
\end{pmatrix} \in \text{SL}_2(\mathbb{Z})$.

\subsection{Dimension formulas}
For integral weights, 
the following lemma gives the dimensions 
of the spaces of modular forms and cusp forms,
see \cite[Théorème 1]{CO1977} or \cite[Theorem 1.34]{Ono2004}.
In this paper, 
$\mathbb{N}$ denotes the set of nonnegative integers,
and $\mathbb{N}^*$ denotes the set of positive integers.

\begin{lemma}\label{lem:1}
	If $k\in \mathbb{Z}$, $N \in \mathbb{N}^*$ and $\chi$ is a Dirichlet character modulo $N$ for which $\chi(-1)=(-1)^{k}$, 
	then
	\begin{align*}
	&\dim \mathcal{S}_{k}(\varGamma_{0}(N),\chi)-\dim \mathcal{M}_{2-k}(\varGamma_{0}(N),\chi) \\
	= &\frac{(k-1)N}{12}\prod_{p\mid N}(1+p^{-1}) 
	-\frac{1}{2}\prod_{p\mid N}\lambda(r_{p},s_{p},p) 
	+\nu_{k}\sum_{\substack{x\pmod{N},\\ x^{2}+1\equiv 0\pmod{N}}}\chi(x)\\
	&+\mu_{k}\sum_{\substack{x\pmod{N},\\ x^{2}+x+1\equiv 0\pmod{N}}}\chi(x),
	\end{align*}
	with the following notation
	\begin{itemize}[leftmargin=1.5em]
		\item 
		If $p \mid N$ is prime, 
		then let $r_p$ (resp. $s_p$) denote the power of $p$ dividing $N$ 
		(resp. the conductor of $\chi$).

		\item 
		$ \lambda(r_p, s_p, p) $ is given by
		\begin{equation*}
			\lambda(r_p,s_p,p) = 
			\begin{cases}
				p^{r'} + p^{r'-1} & \text{if } 2s_p \leqslant r_p = 2r', \\
				2p^{r'} & \text{if } 2s_p \leqslant r_p = 2r' + 1, \\
				2p^{r_p - s_p} & \text{if } 2s_p > r_p.
			\end{cases}
		\end{equation*}

		\item 
		$ \nu_k $ is given by
		\begin{equation*}
		\nu_k = 
		\begin{cases}
		0 & \text{if } k \text{ is odd}, \\
		-\frac{1}{4} & \text{if } k \equiv 2 \pmod{4}, \\
		\frac{1}{4} & \text{if } k \equiv 0 \pmod{4}.
		\end{cases}
		\end{equation*}

		\item 
		$ \mu_k $ is given by
		$$
			\mu_k =
			\begin{cases}
				0 & \text{if } k \equiv 1 \pmod{3}, \\
				-\frac{1}{3} & \text{if } k \equiv 2 \pmod{3}, \\
				\frac{1}{3} & \text{if } k \equiv 0 \pmod{3}.
			\end{cases}
		$$
	\end{itemize}
\end{lemma}

For integers $ k \geqslant 2 $, 
Lemma \ref{lem:1} computes $ \dim \mathcal{M}_{k}(\varGamma_{0}(N),\chi) $
and $\dim \mathcal{S}_{k}(\varGamma_{0}(N),\chi) $ . 
For $ k = 1 $, the dimensions are available from the LMFDB website~\cite{lmfdb}.

For half-integral weights, 
the next lemma gives the dimensions 
of the spaces of modular forms and cusp forms,
see \cite[Théorème 2]{CO1977} or \cite[Theorem 1.56]{Ono2004}.

\begin{lemma}\label{lem:2}
	If $k \in \frac{1}{2} + \mathbb{Z}$, $N \in \mathbb{N}^*,4\mid N$ and $\chi$ is a Dirichlet character modulo $N$ 
	for which $\chi(-1) = 1$, then
	\begin{align*}
		\dim \mathcal{S}_{k}(\varGamma_{0}(N), \chi) - 
		\dim \mathcal{M}_{2-k}(\varGamma_{0}(N), \chi)
		= \frac{(k-1)N}{12} \prod_{p\mid N} (1 + p^{-1}) - 
		\frac{\zeta(k, N, \chi)}{2} \prod_{\substack{p\mid N \\ p \neq 2}} \lambda(r_{p}, s_{p}, p),
	\end{align*}
	where $r_p$, $s_p$, $\lambda(r_p, s_p, p)$ are the same as in Lemma \ref{lem:1}, 
	and where $\zeta(k, N, \chi)$ is defined by the following table
	$$\renewcommand{\arraystretch}{1.3}
	\begin{tabular}{|c|c|c|c|c|}
	\hline
	\multicolumn{4}{|c|}{$ r_2 \geqslant 4 $} & $ \zeta(k, N, \chi) = \lambda (r_2, s_2, 2) $ \\
	\hline
	\multicolumn{4}{|c|}{$ r_2 = 3 $} & $ \zeta(k, N, \chi) = 3 $ \\
	\hline
	\multirow{5}{*}{$ r_2 = 2 $} & \multicolumn{3}{|c|}{Condition (C)} & $ \zeta(k, N, \chi) = 2 $ \\
	\cline{2-5}
	& \multirow{4}{*}{non (C)} &
	\multirow{2}{*}{$ k - \frac{1}{2} \in 2\mathbb{Z} $} & $ s_2 = 0 $ & $ \zeta(k, N, \chi) = \frac{3}{2} $ \\
	\cline{4-5}
	& & & $ s_2 = 2 $ & $ \zeta(k, N, \chi) = \frac{5}{2} $ \\
	\cline{3-5}
	& & \multirow{2}{*}{$ k - \frac{3}{2} \in 2\mathbb{Z} $} & $ s_2 = 0 $ & $ \zeta(k, N, \chi) = \frac{5}{2} $ \\
	\cline{4-5}
	& & & $ s_2 = 2 $ & $ \zeta(k, N, \chi) = \frac{3}{2} $ \\
	\hline
	\end{tabular}
	$$
	Condition (C) is the following condition
	$$
	\text{(C)} \iff \exists p \text{ prime},\ p \equiv 3 \pmod{4},\ p \mid N,\ r_p \text{ odd or } 0 < r_p < 2s_p.
	$$
	Consequently
	$$
	\text{non (C)} \iff \bigl(\forall p \text{ prime}\bigr),\ \bigl(p \equiv 3 \pmod{4} \text{ and } p \mid N \implies r_p \text{ even and } r_p \geq 2s_p\bigr).
	$$
\end{lemma}

For $ k\in \frac{1}{2}+\mathbb{Z}, k \geqslant \frac{5}{2} $, 
Lemma \ref{lem:2} computes $ \dim \mathcal{M}_{k}(\varGamma_{0}(N),\chi) $
and $ \dim \mathcal{S}_{k}(\varGamma_{0}(N),\chi) $.
For $ k = \frac{1}{2},\frac{3}{2} $, 
the following two lemmas are also needed,
see \cite[Theorem A, B]{SS1977} or \cite[Theorem 1.45, 1.46]{Ono2004}.

\begin{lemma}\label{lem:3}
	Suppose that $N \in \mathbb{N}^*,4\mid N$ and that $\chi$ is 
	an even Dirichlet character modulo $N$. 
	Let $\Omega(N,\chi)$ denote the set of pairs $(\psi,t)$, 
	where $t$ is a positive integer, 
	and where $\psi$ is an even primitive Dirichlet character 
	with conductor $r(\psi)$ satisfying the following
	\begin{enumerate}
		\item We have $4r(\psi)^{2}t \mid N$.
		\item We have $\chi(n) = \psi(n)\left(\frac{t}{n}\right)$ for every integer $n$ coprime to $N$.
	\end{enumerate}
	Then the theta functions 
	$\theta_{\psi,t}(z)=\sum \limits_{n=-\infty}^{\infty}\psi(n)q^{tn^2}$, 
	with $(\psi,t) \in \Omega(N,\chi)$ ,
	make up a basis of $\mathcal{M}_{\frac{1}{2}}(\varGamma_{0}(N),\chi)$.
\end{lemma}

Every Dirichlet character $\psi$ of conductor $r(\psi)$ 
may be written uniquely as
$\psi = \prod \limits_{p \mid r(\psi)} \psi_p,$
where $\psi_p$ is a Dirichlet character 
whose conductor is the highest power of the prime $p$ dividing $r(\psi)$. 
We say that $\psi$ is \emph{totally even} if $\psi_p(-1) = 1$ 
for every prime $p \mid r(\psi)$.

\begin{lemma}\label{lem:4}
	Suppose that $N \in \mathbb{N}^*,4\mid N$ and that $\chi$ is 
	an even Dirichlet character modulo $N$. 
	The set of theta functions $\theta_{\psi, t}$, as $(\psi, t)$ varies 
	over the elements $\Omega(N, \chi)$ 
	for which $\psi$ is not totally even, forms a basis of 
	$\mathcal{S}_{\frac{1}{2}}(\varGamma_0(N), \chi)$,
	where $\theta_{\psi, t}$, $\Omega(N, \chi)$ are the same as in Lemma \ref{lem:3}. 
\end{lemma}

\subsection{Operators on modular forms}
If $f(z)=\sum _{n=0}^{\infty} c(n)q^n$, $d\in \mathbb{N}^*$, 
then the $V$-operator $V(d)$ is defined by
$$
f(z)\mid V(d) = \sum_{n = 0}^{\infty} c(n) q^{dn}, \qquad f\mid V(d) (z)=f(dz).
$$
Define the $U$-operator $U(d)$ by
$$
f(z)\mid  U(d)= \sum_{n =0}^{\infty} c(dn) q^n.
$$
If $\psi$ is a Dirichlet character, then the $\psi$-twist of $f(z)$ is defined by
$$f_{\psi}(z) = f(z) \otimes \psi=\sum_{n=0}^{\infty}\psi(n)c(n)q^n.$$
If $M\in \mathbb{N}^*$, $m \in \mathbb{Z}$,
define the operator $S_{M,m}$ by
$$
f(z)\mid  S_{M,m}= \sum_{\substack{n \in \mathbb{N}\\ n\equiv m\pmod{M}}} c(n) q^n.
$$
To describe the action of these operators on modular forms, 
we introduce $ \operatorname{rad}(n) = \prod_{p \mid n} p $
and let $ \operatorname{dis}(n) $ denote 
the discriminant of $ \mathbb{Q}(\sqrt{n}) $.
\begin{lemma}\label{lem:5}
	Let $k \in \mathbb{Z}$, $N \in \mathbb{N}^*$ ,
	$\chi$ be a Dirichlet character modulo $N$,
	and $f(z) \in \mathcal{M}_k(\varGamma_0(N), \chi)$
	(resp. $\mathcal{S}_k(\varGamma_0(N), \chi)$).
	\begin{enumerate}
	\item If $d\in \mathbb{N}^*$, then
	$f(z)\mid V(d) \in \mathcal{M}_k(\varGamma_0(Nd), \chi)$
	(resp. $\mathcal{S}_k(\varGamma_0(Nd), \chi)$).

	\item If $d\in \mathbb{N}^*$ and $\operatorname{rad}(d) \mid N$, then
	$f(z) \mid U(d) \in \mathcal{M}_k(\varGamma_0(N), \chi)$
	(resp. $\mathcal{S}_k(\varGamma_0(N), \chi)$).

	\item If $\psi$ is a Dirichlet character with modulus $m$, then
	$f_{\psi}(z) \in \mathcal{M}_k(\varGamma_0(Nm^2), \chi\psi^2)$
	(resp. $\mathcal{S}_k(\varGamma_0(Nm^2), \chi\psi^2)$).
	In particular, if $\psi$ is quadratic,
	then $f_{\psi}(z) \in \mathcal{M}_k(\varGamma_0(Nm^2), \chi)$.

	\item If $A,B \in \mathbb{N}^*$, $C \in \mathbb{Z}$, $\gcd(B,C)=1$, $B \mid 24$,
	and let $N_1=\operatorname{lcm}(N,\operatorname{rad}(A))$,
	then $f(z)\mid S_{AB,AC} \in 
	\mathcal{M}_k(\varGamma_0(N_1AB^2), \chi)$
	(resp. $\mathcal{S}_k(\varGamma_0(N_1AB^2), \chi)$).
	\end{enumerate}
\end{lemma}

\begin{lemma}\label{lem:6}
	Let $k \in \frac{1}{2}+\mathbb{Z}$, $N \in \mathbb{N}^*$, $4\mid N$,
	$\chi$ be a Dirichlet character modulo $N$,
	and $f(z) \in \mathcal{M}_k(\varGamma_0(N), \chi)$
	(resp. $\mathcal{S}_k(\varGamma_0(N), \chi)$).
	\begin{enumerate}
	\item If $d\in \mathbb{N}^*$, then
	$f(z)\mid  V(d) \in \mathcal{M}_k(\varGamma_0(Nd), \chi\chi_d)$
	(resp. $\mathcal{S}_k(\varGamma_0(Nd), \chi\chi_d)$).

	\item If $d\in \mathbb{N}^*$, $\operatorname{rad}(d) \mid  N$, 
	and $\operatorname{dis}(d) \mid  N$,
	then $f(z) \mid  U(d) \in \mathcal{M}_k(\varGamma_0(N), \chi\chi_d)$
	(resp. \break $\mathcal{S}_k(\varGamma_0(N), \chi\chi_d)$).

	\item If $\psi$ is a Dirichlet character with modulus $m$, then
	$f_{\psi}(z) \in \mathcal{M}_k(\varGamma_0(Nm^2), \chi\psi^2)$
	(resp. $\mathcal{S}_k(\varGamma_0(Nm^2), \chi\psi^2)$).
	In particular, if $\psi$ is quadratic,
	then $f_{\psi}(z) \in \mathcal{M}_k(\varGamma_0(Nm^2), \chi)$.

	\item If $A,B \in \mathbb{N}^*$, $C \in \mathbb{Z}$, $\gcd(B,C)=1$, $B \mid 24$,
	and let $N_1=\operatorname{lcm}(N,\operatorname{rad}(A),\operatorname{dis}(A))$,
	then $f(z)\mid S_{AB,AC} \in 
	\mathcal{M}_k(\varGamma_0(N_1AB^2), \chi)$
	(resp. $\mathcal{S}_k(\varGamma_0(N_1AB^2), \chi)$).
	\end{enumerate}
\end{lemma}

Parts (1)--(3) of Lemma~\ref{lem:5} and Lemma~\ref{lem:6} are well known. 
Moreover, part~(4) of Lemma~\ref{lem:5} and part~(4) of Lemma~\ref{lem:6} 
follow from parts~(1)--(3) of the respective lemmas.
We now prove Lemma~\ref{lem:5} (4). 
The proof of Lemma~\ref{lem:6} (4) is identical.

\begin{proof}[Proof of Lemma~\ref{lem:5} (4)]
	We have the identity
	$$
	f(z) \mid S_{AB,AC} = f(z) \mid U(A) \mid S_{B,C} \mid V(A).
	$$
	Then by Lemma~\ref{lem:5}(2), it follows that
	$$
	f(z) \mid U(A) \in \mathcal{M}_k(\varGamma_0(N_1), \chi).
	$$
	Since $B$ and $C$ are coprime, 
	and by the orthogonality relations of Dirichlet characters,
	we have
	$$
	f(z) \mid U(A) \mid S_{B,C} = \frac{1}{\varphi(B)} \sum_{\psi \in \widehat{(\mathbb{Z}/B\mathbb{Z})^{*}}} \overline{\psi}(C) \cdot \bigl( f(z) \mid U(A) \bigr) \otimes \psi,
	$$
	where $\varphi(B)$ is Euler's totient function and $\widehat{(\mathbb{Z}/B\mathbb{Z})^{*}}$ denotes the group of characters modulo $B$.
	When $ B \mid 24 $, 
	every character in $ \widehat{(\mathbb{Z}/B\mathbb{Z})^{*}} $ is quadratic. Hence, 
	by Lemma~\ref{lem:5}(3), 
	we have
	$$
	f(z) \mid U(A) \mid S_{B,C} \in \mathcal{M}_k(\varGamma_0(N_1 B^2), \chi).
	$$
	Finally, by Lemma~\ref{lem:5}(1), 
	$$
	f(z) \mid S_{AB,AC} = f(z) \mid U(A) \mid S_{B,C} \mid V(A) \in \mathcal{M}_k(\varGamma_0(N_1 A B^2), \chi).
	$$
\end{proof}

\subsection{Sturm's theorem}
In order to decide whether two modular forms are equal, 
we require the following theorem of Sturm \cite{Sturm1987}.
If $ f(z) = \sum_{n=0}^{\infty} a(n) q^n $, 
define the order of $ f $ at $ \infty $ to be
$\operatorname{ord}_{\infty}(f) = \min\{ n \mid a(n) \neq 0 \}$
except when $ f = 0 $, 
in which case $ \operatorname{ord}_{\infty}(f) = \infty $.

\begin{lemma}\label{lem:7}
	Let $k \in \frac{1}{2}\mathbb{Z}$, $N \in \mathbb{N}^*$
	(if $k\in \frac{1}{2}+\mathbb{Z}$, than $4 \mid N$),
	$\chi$ be a Dirichlet character modulo $N$,
	and $f \in \mathcal{M}_k(\varGamma_0(N), \chi)$.
	If $\operatorname{ord}_{\infty}(f)> \frac{kN}{12}\prod_{p\mid N}(1+p^{-1})$,
	then $f = 0$.
\end{lemma}

\subsection{Modular transformation, and holomorphy criteria for eta quotients}
The next two lemmas are used to determine whether an eta quotient is a modular form,
see \cite[Theorem 1.64]{Ono2004} and \cite[Corollary 2.3]{kohler2011eta}.

\begin{lemma}\label{lem:8}
	If $f(z)=\prod_{\delta \mid N} \eta(\delta z)^{r_{\delta}}$
	is an eta quotient with 
	$k = \frac{1}{2} \sum_{\delta \mid N} r_{\delta} \in \mathbb{Z} $,
	with the additional properties that
	$$
	\sum_{\delta \mid N} \delta r_{\delta} \equiv 0 \pmod{24}
	$$
	and
	$$
	\sum_{\delta \mid N} \frac{N}{\delta} r_{\delta} \equiv 0 \pmod{24},
	$$
	then $ f(z) $ satisfies
	$$
	f\left( \frac{az + b}{cz + d} \right) = \left(\frac{(-1)^k s}{d}\right)(cz + d)^k f(z)
	$$
	for every $ \begin{pmatrix} a & b \\ c & d \end{pmatrix} \in \varGamma_0(N) $,
	here $s=\prod_{\delta \mid N}\delta^{|r_{\delta}|} $.
\end{lemma}

\begin{lemma}\label{lem:9}
	An eta quotient
	$f(z)=\prod_{\delta \mid N} \eta(\delta z)^{r_{\delta}}$ is holomorphic
	if and only if the inequalities
	$$\sum_{\delta \mid N} 
	\frac{(\operatorname{gcd}(d, \delta))^{2}}{\delta} {r_{\delta}}\geqslant 0$$
	hold for all positive divisors $d$ of $N$. 
	It is a cuspidal etaquotient if and only
	if all these inequalities hold strictly.
\end{lemma}

\section{Proof of Theorem~\ref{tho:1}}
\begin{proof}[Proof of Theorem~\ref{tho:1}]
	(1) Let $f_1(z)=\frac{\eta(2z)^5}{\eta(z)^2\eta(4z)^2} , 
	f_2(z)= \frac{\eta(4z)^5}{\eta(2z)^2\eta(8z)^2}$.
	By Lemma~\ref{lem:8},\ref{lem:9},
	we have $f_1\in \mathcal{M}_{1/2}(\varGamma_0(4) ) $,
	$f_2\in \mathcal{M}_{1/2}(\varGamma_0(8),\chi_8 ) $.

	(2) $f_1$ and $f_2$ are linearly independent 
	because they belong to spaces of modular forms with distinct Dirichlet characters.
	Moreover, 
	since $\mathbb{C}$ is algebraically closed, 
	the set $\{ ( \frac{f_1}{f_2} )^k \mid k \in \mathbb{N} \}$ 
	is linearly independent.
	Furthermore, 
	for a given $k \in \mathbb{N}^*$, 
	the set $\{ f_1^{\alpha_1} f_2^{\alpha_2} \mid \alpha_1 + \alpha_2 = k,\; \alpha_1,\alpha_2 \in \mathbb{N} \}$ 
	is linearly independent.
	Finally, 
	because modular forms of distinct weights are linearly independent, 
	we conclude that $f_1$ and $f_2$ are algebraically independent.

	(3) By Lemmas~\ref{lem:1}--\ref{lem:4}, for every half-integer $k \geqslant 0$ we have
	$$
	\dim \mathcal{M}_k^{\mathrm{quad}}(\varGamma_0(8)) = 2k+1,
	\qquad
	\dim \mathcal{S}_k^{\mathrm{quad}}(\varGamma_0(8)) = 
	\begin{cases}
	2k-5 & \text{if } k \geqslant 3,\\
	0   & \text{if } 0 \leqslant k < 3.
	\end{cases}
	$$
	Comparing dimensions of $\mathcal{M}_k^{\mathrm{quad}}(\varGamma_0(8))$ 
	and $\mathbb{C}[f_1,f_2] \bigcap \mathcal{M}_k^{\mathrm{quad}}(\varGamma_0(8))$
	for each $k\in \frac{1}{2}\mathbb{Z}  $, 
	we have $\mathcal{M}_{\mathbb{Z}/2}^{\mathrm{quad}}(\varGamma_0(8))=\mathbb{C}[f_1,f_2]$.
	Thus $\mathcal{M}_{\mathbb{Z}/2}^{\mathrm{quad}}(\varGamma_0(8))$ is generated as an algebra 
	by $\frac{\eta(2z)^5}{\eta(z)^2\eta(4z)^2}$ and $\frac{\eta(4z)^5}{\eta(2z)^2\eta(8z)^2}$.
	
	(4) By Lemma~\ref{lem:8},\ref{lem:9},
	we have 
	$g(z)=\eta(z)^2\eta(2z)\eta(4z)\eta(8z)^2\in \mathcal{S}_{3}(\varGamma_0(8),\chi_{-8} ) $.
	Computing the first three Fourier coefficients of both sides of
	$g = \frac{1}{4}\bigl(-2f_1f_2^5 + 3f_1^3f_2^3 - f_1^5f_2\bigr)$
	and applying Lemma~\ref{lem:7} yields the identity.
	Comparing dimensions of $\mathcal{S}_k^{\mathrm{quad}}(\varGamma_0(8))$ 
	and \break 
	$(g\cdot\mathbb{C}[f_1,f_2]) \bigcap \mathcal{S}_k^{\mathrm{quad}}(\varGamma_0(8))$
	for each $k\in \frac{1}{2}\mathbb{Z}  $, 
	we have $\mathcal{S}_{\mathbb{Z}/2}^{\mathrm{quad}}(\varGamma_0(8))=g \cdot \mathbb{C}[f_1,f_2]$.
\end{proof}

\section{Proof of Theorem~\ref{tho:2}}
The proof of Theorem~\ref{tho:2} is similar to that of Theorem~\ref{tho:1}, 
but due to the presence of an additional ideal $I_{12}$, 
some extra techniques are required.

\begin{proof}[Proof of Theorem~\ref{tho:2}]
	(1) Let the four eta quotients
	$\frac{\eta(2z)^5}{\eta(z)^2\eta(4z)^2}$,
	$\frac{\eta(6z)^5}{\eta(3z)^2\eta(12z)^2}$,
	$\frac{\eta(4z)^2\eta(12z)^2}{\eta(2z)\eta(6z)}$,\break
	$\frac{\eta(2z)\eta(12z)^6}{\eta(4z)^2\eta(6z)^3}$
	be respectively denoted by $f_1$, $f_2$, $f_3$, $f_4$.
	By Lemma~\ref{lem:8},\ref{lem:9},
	we have \break
	$f_1\in \mathcal{M}_{1/2}(\varGamma_0(4) ) $,
	$f_2\in \mathcal{M}_{1/2}(\varGamma_0(12),\chi_{12} ) $,
	$f_3,f_4\in \mathcal{M}_{1}(\varGamma_0(12),\chi_{-3} ) $.
	Consequently, \break
	$\mathbb{C}[f_1,f_2,f_3,f_4] \subset \mathcal{M}_{\mathbb{Z}/2}^{\mathrm{quad}}(\varGamma_0(12))$.
	Define the map
	$$
	\phi_{12}: \mathbb{C}[x_1,x_2,x_3,x_4] \longrightarrow \mathbb{C}[f_1,f_2,f_3,f_4],\quad
	x_i \mapsto f_i\ (i=1,2,3,4).
	$$
	Using Lemma~\ref{lem:7} together with the computation of 
	sufficiently many Fourier coefficients, 
	we obtain $I_{12} \subset \ker \phi_{12}$.
	Consequently, we have the induced map
	$$
	\bar{\phi}_{12} : \mathbb{C}[x_1,x_2,x_3,x_4]/I_{12} \longrightarrow \mathbb{C}[f_1,f_2,f_3,f_4],\quad
	x_i+I_{12} \mapsto f_i\ (i=1,2,3,4).
	$$

	(2) We choose a weighted graded reverse lexicographic order 
	on $\mathbb{C}[x_1,x_2,x_3,x_4]$ 
	with \break
	$\deg_\text{w}(x^{\alpha}) = \alpha_1 + \alpha_2 + 2\alpha_3 + 2\alpha_4$ 
	and $x_1 > x_2 > x_3 > x_4$.
	Under this order, 
	$\{x_1^2 x_3 - x_2^2 x_3 - x_1^2 x_4 - 3x_2^2 x_4,
	x_3^2 - x_1x_2 x_4 + 2 x_3 x_4 - 3 x_4^2,
	x_1^3 - x_1 x_2^2 - 4 x_2 x_3 - 12 x_2 x_4,
	x_1^2 x_2 - x_2^3 - 4 x_1 x_3 + 4 x_1 x_4\}$
	is a Gr\"{o}bner basis for $I_{12}$.
	Furthermore, 
	we obtain 
	$\operatorname{LT}(I_{12}) = 
	\langle x_1^2 x_3, x_3^2, x_1^3, x_1^2 x_2\rangle$.
	Let $ S = \{ x^{\alpha} + I_{12} \mid x^{\alpha} 
	\notin \operatorname{LT}(I_{12}) \} $. 
	Then $ S $ is a basis of $ \mathbb{C}[x_1,x_2,x_3,x_4]/I_{12} $.

	We want to prove that $\bar{\phi}_{12}$ is an isomorphism. 
	It suffices to show that $\bar{\phi}_{12}(S)$ is 
	a basis of $\mathbb{C}[f_1,f_2,f_3,f_4]$.
	Given $k\in\frac{1}{2}\mathbb{Z},k>0$,
	let
	$$
		A_{k,0}=\begin{cases}
			\{f_1^{\alpha_1}f_2^{\alpha_2}f_3^{\alpha_3}f_4^{\alpha_4} 
			\mid (C),\; \alpha_2+\alpha_3+\alpha_4 \equiv 0 \!\!\!\!\pmod 2\}
			& \text{if}\; k\notin 1+2\mathbb{Z},\\

			\{f_1^{\alpha_1}f_2^{\alpha_2}f_3^{\alpha_3}f_4^{\alpha_4} 
			\mid (C),\; \alpha_2+\alpha_3+\alpha_4 \equiv 0 \!\!\!\!\pmod 2\}
			\cup \{f_1^2 f_4^{k-1}\}
			& \text{if}\; k\in 1+2\mathbb{Z},\\
		\end{cases}
	$$
	and 
	$$
		A_{k,1}=\begin{cases}
			\{f_1^{\alpha_1}f_2^{\alpha_2}f_3^{\alpha_3}f_4^{\alpha_4} 
			\mid (C),\; \alpha_2+\alpha_3+\alpha_4 \equiv 1 \!\!\!\!\pmod 2\}
			& \text{if}\; k\notin 2\mathbb{Z},\\

			\{f_1^{\alpha_1}f_2^{\alpha_2}f_3^{\alpha_3}f_4^{\alpha_4} 
			\mid (C),\; \alpha_2+\alpha_3+\alpha_4 \equiv 1 \!\!\!\!\pmod 2\}
			\cup \{f_1^2 f_4^{k-1}\}
			& \text{if}\; k\in 2\mathbb{Z},\\
		\end{cases}
	$$
	where (C) is the following condition
	$$
	\text{(C)} \iff \alpha_1+\alpha_2+2\alpha_3+2\alpha_4=2k,
			\alpha_1,\alpha_3\in \{0,1\},
			\alpha_2,\alpha_4 \in \mathbb{N}.
	$$
	By $\operatorname{LT}(I_{12}) = 
	\langle x_1^2 x_3, x_3^2, x_1^3, x_1^2 x_2\rangle$,
	we obtain
	$$
		\bar{\phi}_{12}(S)=\{1\}\bigsqcup
		\left(\bigsqcup_{k\in\frac{1}{2}\mathbb{Z},k>0,i\in\{0,1\}} A_{k,i}\right).
	$$
	Moreover, since
	$$
	A_{k,0} \subset 
	\begin{cases}
	\mathcal{M}_k(\varGamma_0(12)) & \text{if } k \notin 1+2\mathbb{Z},\\[4pt]
	\mathcal{M}_k(\varGamma_0(12),\chi_{-4}) & \text{if } k \in 1+2\mathbb{Z},
	\end{cases}
	\quad
	A_{k,1} \subset 
	\begin{cases}
	\mathcal{M}_k(\varGamma_0(12),\chi_{12}) & \text{if } k \notin 1+2\mathbb{Z},\\[4pt]
	\mathcal{M}_k(\varGamma_0(12),\chi_{-3}) & \text{if } k \in 1+2\mathbb{Z},
	\end{cases}
	$$
	and because modular forms with different weights 
	or different Dirichlet characters are linearly independent, 
	to prove that $\bar{\phi}_{12}(S)$ is a basis 
	it suffices to show that the $A_{k,i}$ is linearly independent.

	Computing the Fourier coefficients, we obtain
	$\operatorname{ord}_{\infty}(f_1) = \operatorname{ord}_{\infty}(f_2) = 0$,
	$\operatorname{ord}_{\infty}(f_3) = 1$,
	$\operatorname{ord}_{\infty}(f_4) = 2$.
	Consequently,
	$
	\operatorname{ord}_{\infty}
	(f_1^{\alpha_1} f_2^{\alpha_2} f_3^{\alpha_3} f_4^{\alpha_4}) 
	= \alpha_3 + 2\alpha_4.
	$
	For $ k \in \frac{1}{2}\mathbb{Z} $, $ k > 0 $, 
	and $ k \notin 1+2\mathbb{Z} $, 
	it is easy to verify that any two modular forms in $ A_{k,0} $ 
	have distinct orders at $ \infty $. 
	Therefore $ A_{k,0} $ is linearly independent.
	For $ k \in 1+2\mathbb{Z} $, $ k > 0 $, 
	among the elements of $ A_{k,0} $, 
	only $ f_1^2 f_4^{k-1} $ and $ f_2^2 f_4^{k-1} $ 
	have the same order at $ \infty $, 
	and this order is the largest in $ A_{k,0} $. 
	Moreover, 
	$ f_1^2 $ and $ f_2^2 $ are linearly independent. 
	Therefore $ A_{k,0} $ is linearly independent.
	The situation for $ A_{k,1} $ is almost the same as that for $ A_{k,0} $. 
	Hence we obtain that $ \bar{\phi}_{12} $ is an isomorphism.

	(3) By Lemmas~\ref{lem:1}--\ref{lem:4}, for every half-integer $k \geqslant 0$ we have
	$$
	\dim \mathcal{M}_k^{\mathrm{quad}}(\varGamma_0(12)) = 
	\begin{cases}
	4k+1 & \text{if } k \in \mathbb{Z},\\
	4k   & \text{if } k \in \frac{1}{2}+\mathbb{Z},
	\end{cases}
	$$
	$$
	\dim \mathcal{S}_k^{\mathrm{quad}}(\varGamma_0(12)) = 
	\begin{cases}
	0   & \text{if } 0 \leqslant k \leqslant 2,\\
	4k-9 & \text{if } k \in \mathbb{Z},k>2,\\
	4k-8 & \text{if } k \in \frac{1}{2}+\mathbb{Z},k>2.
	\end{cases}
	$$
	By counting the number of elements in $A_{k,0}$ and $A_{k,1}$, 
	or by using the fact that $\bar{\phi}_{12}$ is an isomorphism, 
	we obtain that for each $k \in \frac{1}{2}\mathbb{Z}$ 
	the dimensions of $\mathcal{M}_k^{\mathrm{quad}}(\varGamma_0(12))$ 
	and $\mathbb{C}[f_1,f_2,f_3,f_4] \cap 
	\mathcal{M}_k^{\mathrm{quad}}(\varGamma_0(12))$ are equal.
	Therefore, 
	$\mathcal{M}_{\mathbb{Z}/2}^{\mathrm{quad}}(\varGamma_0(12))=\mathbb{C}[f_1,f_2,f_3,f_4]$.

	(4) By Lemma~\ref{lem:8},\ref{lem:9},
	we have 
	$g_1(z)=\frac{\eta(2z)^3\eta(3z)^2\eta(12z)^2}{\eta(6z)^2}\in \mathcal{S}_{5/2}(\varGamma_0(12)) $,
	$g_2(z)= \break
	\frac{\eta(z)^2\eta(4z)^2\eta(6z)^3}{\eta(2z)^2}\in \mathcal{S}_{5/2}(\varGamma_0(12),\chi_{12}) $.
	By Lemma~\ref{lem:7},
	$g_1=\phi_{12}(x_2^3x_3-4x_1x_3x_4)$,
	$g_2= \break
	\phi_{12}(x_1x_2^2x_3-4x_1x_2^2x_4+8x_2x_3x_4-12x_2x_4^2)$.
	Therefore, 
	$\bar{\phi}_{12} (J_{12}/I_{12})\subset 
	\mathcal{S}_{\mathbb{Z}/2}^{\mathrm{quad}}(\varGamma_0(12))$.
	Comparing dimensions of $\mathcal{S}_k^{\mathrm{quad}}(\varGamma_0(12))$ 
	and 
	$\bar{\phi}_{12} (J_{12}/I_{12}) \bigcap \mathcal{S}_k^{\mathrm{quad}}(\varGamma_0(12))$
	for each $k\in \frac{1}{2}\mathbb{Z}  $, 
	we have 
	$\mathcal{S}_{\mathbb{Z}/2}^{\mathrm{quad}}(\varGamma_0(12)) \cong J_{12}/I_{12}$.

\end{proof}

\section{Proof of Theorem~\ref{tho:3}--\ref{tho:5}}

The proofs of Theorem~\ref{tho:3}--\ref{tho:5}
	are similar to that of Theorem~\ref{tho:1} and Theorem~\ref{tho:2}. 
	Therefore we only present the key steps.

\begin{proof}[Sketch of proof of Theorem~\ref{tho:3}]
	(1) Let the three eta quotients
	$\frac{\eta(2z)^5}{\eta(z)^2\eta(4z)^2}$,
	$\frac{\eta(4z)^5}{\eta(2z)^2\eta(8z)^2}$,
	$\frac{\eta(16z)^2}{\eta(8z)}$
	be respectively denoted by $f_1$, $f_2$, $f_3$.
	By Lemma~\ref{lem:8},\ref{lem:9},
	we have $f_1\in \mathcal{M}_{1/2}(\varGamma_0(4) ) $,
	$f_2\in \break
	\mathcal{M}_{1/2}(\varGamma_0(8),\chi_{8} ) $,
	$f_3\in \mathcal{M}_{1/2}(\varGamma_0(16)) $.
	Let $I_{16}=\langle x_1^2-4x_1x_3-x_2^2+8x_3^2\rangle$.
	Define the map
	$$
	\bar{\phi}_{16} : \mathbb{C}[x_1,x_2,x_3]/I_{16} \longrightarrow \mathbb{C}[f_1,f_2,f_3],\quad
	x_i+I_{16} \mapsto f_i\ (i=1,2,3).
	$$

	(2) We choose a graded reverse lexicographic order 
	on $\mathbb{C}[x_1,x_2,x_3]$ 
	with $\deg(x^{\alpha}) = \alpha_1 + \alpha_2 + \alpha_3$ 
	and $x_1 > x_2 > x_3$.
	Under this order, 
	$\{x_1^2-4x_1x_3-x_2^2+8x_3^2\}$
	is a Gr\"{o}bner basis for $I_{16}$.
	Furthermore, 
	we obtain 
	$\operatorname{LT}(I_{16}) = 
	\langle x_1^2\rangle$.
	Let $ S = \{ x^{\alpha} + I_{16} \mid x^{\alpha} 
	\notin \operatorname{LT}(I_{16}) \} $. 

	Given $k\in\frac{1}{2}\mathbb{Z},k>0$,
	let
	$$
		A_{k,0}=
			\{f_1^{\alpha_1}f_2^{\alpha_2}f_3^{\alpha_3}
			\mid \alpha_1+\alpha_2+\alpha_3=2k,
			\alpha_1\in \{0,1\},
			\alpha_2,\alpha_3 \in \mathbb{N},\; \alpha_1+\alpha_3 \equiv 0 \!\!\!\!\pmod 2\},
	$$
	$$
		A_{k,1}=
			\{f_1^{\alpha_1}f_2^{\alpha_2}f_3^{\alpha_3}
			\mid \alpha_1+\alpha_2+\alpha_3=2k,
			\alpha_1\in \{0,1\},
			\alpha_2,\alpha_3 \in \mathbb{N},\; \alpha_1+\alpha_3 \equiv 1 \!\!\!\!\pmod 2\}.
	$$
	By $\operatorname{LT}(I_{16}) = 
	\langle x_1^2\rangle$,
	we obtain
	$$
		\bar{\phi}_{16}(S)=\{1\}\bigsqcup
		\left(\bigsqcup_{k\in\frac{1}{2}\mathbb{Z},k>0,i\in\{0,1\}} A_{k,i}\right).
	$$
	Moreover, since
	$$
	A_{k,0} \subset 
	\begin{cases}
	\mathcal{M}_k(\varGamma_0(16),\chi_8) & \text{if } k \in \frac{1}{2}+\mathbb{Z},\\
	\mathcal{M}_k(\varGamma_0(16),\chi_{-4}) & \text{if } k \in 1+2\mathbb{Z},\\
	\mathcal{M}_k(\varGamma_0(16)) & \text{if } k \in 2\mathbb{Z},
	\end{cases}
	\quad
	A_{k,1} \subset 
	\begin{cases}
	\mathcal{M}_k(\varGamma_0(16)) & \text{if } k \in \frac{1}{2}+\mathbb{Z},\\
	\mathcal{M}_k(\varGamma_0(16),\chi_{-8}) & \text{if } k \in 1+2\mathbb{Z},\\
	\mathcal{M}_k(\varGamma_0(16),\chi_{8}) & \text{if } k \in 2\mathbb{Z}.
	\end{cases}
	$$

	Computing the Fourier coefficients, we obtain
	$\operatorname{ord}_{\infty}(f_1) = \operatorname{ord}_{\infty}(f_2) = 0$,
	$\operatorname{ord}_{\infty}(f_3) = 1$.
	Consequently,
	$
	\operatorname{ord}_{\infty}
	(f_1^{\alpha_1} f_2^{\alpha_2} f_3^{\alpha_3}) 
	= \alpha_3.
	$
	For $ k \in \frac{1}{2}\mathbb{Z} $, $ k > 0 $, 
	it is easy to verify that any two modular forms in $ A_{k,0} $ 
	have distinct orders at $ \infty $. 
	Therefore $ A_{k,0} $ is linearly independent.
	The situation for $ A_{k,1} $ is almost the same as that for $ A_{k,0} $. 
	Hence we obtain that $ \bar{\phi}_{16} $ is an isomorphism.

	(3) By Lemmas~\ref{lem:1}--\ref{lem:4}, for every half-integer $k \geqslant 0$ we have
	$$
	\dim \mathcal{M}_k^{\mathrm{quad}}(\varGamma_0(16)) = 
	4k+1,\quad
	\dim \mathcal{S}_k^{\mathrm{quad}}(\varGamma_0(16)) = 
		\begin{cases}
		0   & \text{if } 0 \leqslant k \leqslant 2,\\
		4k-9 & \text{if } k>2.
		\end{cases}
	$$
\end{proof}

\begin{proof}[Sketch of proof of Theorem~\ref{tho:4}]
	(1) Let these eta quotients
	$\frac{\eta(2z)^5}{\eta(z)^2\eta(4z)^2}$,
	$\frac{\eta(10z)^5}{\eta(5z)^2\eta(20z)^2}$,\break
	$\frac{\eta(z)\eta(2z)\eta(10z)\eta(20z)}{\eta(4z)\eta(5z)}$,
	$\frac{\eta(2z)^5\eta(5z)^2\eta(20z)^7}{\eta(z)^2\eta(4z)^3\eta(10z)^5}$,
	$\frac{\eta(20z)^8}{\eta(10z)^4}$,
	$\frac{\eta(20z)^5}{\eta(4z)}$,
	$\frac{\eta(2z)\eta(20z)^{10}}{\eta(4z)^2\eta(10z)^5}$,
	be respectively denoted by $f_1,f_2,\cdots,f_7$.
	By Lemma~\ref{lem:8},\ref{lem:9},
	we have $f_1\in \mathcal{M}_{1/2}(\varGamma_0(4) ) $,
	$f_2\in \mathcal{M}_{1/2}(\varGamma_0(20),\chi_{5} ) $,
	$f_3 \in \mathcal{M}_{1}(\varGamma_0(20),\chi_{-5}) $,
	$f_4,f_5\in \mathcal{M}_{2}(\varGamma_0(20))$,
	$f_6,f_7\in \mathcal{M}_{2}(\varGamma_0(20),\chi_{5}) $.
	Define the map
	$$
	\bar{\phi}_{20} : \mathbb{C}[x_1,x_2,\cdots,x_7]/I_{20} \longrightarrow \mathbb{C}[f_1,f_2,\cdots,f_7],\quad
	x_i+I_{20} \mapsto f_i\ (i=1,2,\cdots,7).
	$$
	
	(2) We choose a graded reverse lexicographic order 
	on $\mathbb{C}[x_1,x_2,\cdots,x_7]$ 
	with $\deg_{\text{w}}(x^{\alpha}) = \alpha_1 + \alpha_2 + 2\alpha_3
	+4(\alpha_4+\alpha_5+\alpha_6+\alpha_7)$ 
	and $x_1 > x_2 > x_3 > x_4 > x_6 > x_5 > x_7 $.
	Under this order, 
	we obtain 
	$\operatorname{LT}(I_{20}) = 
	\langle x_{1}^{2} x_{3}^{3},x_{3}^{4},x_{4}^{2}, x_{4} x_{6}, x_{6}^{2}, x_{4} x_{5}, x_{6} x_{5}, x_{5}^{2}, 
	x_{1}^{3} x_{3}^{2}, 
	x_{3} x_{4}, x_{1}^{2} x_{6}, x_{3} x_{6}, x_{1}^{2} x_{5}, x_{3} x_{5}, 
	x_{1}^{2} x_{2}^{3},\allowbreak x_{1}^{2} x_{2} x_{3},\allowbreak x_{1} x_{4},\allowbreak x_{2} x_{4},\allowbreak
	x_{1}^{4}\rangle$.
	Let $ S = \{ x^{\alpha} + I_{20} \mid x^{\alpha} 
	\notin \operatorname{LT}(I_{20}) \} $. 

	Given $k\in\frac{1}{2}\mathbb{Z},k>0$,
	let
	$$
		S_{k,0}=S\cap
		\{x^{\alpha}+I_{20}\mid \deg_{\text{w}}(x^{\alpha})=2k, 
		\alpha_2+\alpha_3+\alpha_6+\alpha_7 \equiv 2k \!\!\!\!\pmod 2 \},
	$$
	$$
		S_{k,1}=S\cap
		\{x^{\alpha}+I_{20}\mid \deg_{\text{w}}(x^{\alpha})=2k, 
		\alpha_2+\alpha_3+\alpha_6+\alpha_7 \equiv 2k+1 \!\!\!\!\pmod 2 \}.
	$$
	Let $\bar{\phi}_{20}(S_{k,i})=A_{k,i}$,
	we obtain
	$$
		S=\{1+I_{20} \}\bigsqcup
		\left(\bigsqcup_{k\in\frac{1}{2}\mathbb{Z},k>0,i\in\{0,1\}} S_{k,i}\right),\qquad
		\bar{\phi}_{20}(S)=\{1\}\bigsqcup
		\left(\bigsqcup_{k\in\frac{1}{2}\mathbb{Z},k>0,i\in\{0,1\}} A_{k,i}\right).
	$$
	For $k\in \frac{1}{2}\mathbb{Z}, k>4$,
	let
	$B_{k,0}=\{f_2^{2k},f_1f_2^{2k-3}f_3,f_2^{2k-4}f_3^2,f_1f_2^{2k-7}f_3^3,f_1f_2^{2k-5}f_6,f_2^{2k-4}f_5 \} $
	and
	$B_{k,1}=\{f_1f_2^{2k-1},f_2^{2k-2}f_3,\allowbreak f_1f_2^{2k-5}f_3^2,f_2^{2k-6}f_3^3,f_2^{2k-4}f_6,f_1f_2^{2k-5}f_5 \} $,
	we obtain
	$$A_{k,0}=f_7 A_{k-2,1} \sqcup B_{k,0},\quad  A_{k,1}=f_7 A_{k-2,0} \sqcup B_{k,1}. $$
	Moreover, since
	$$
	A_{k,0} \subset 
	\begin{cases}
	\mathcal{M}_k(\varGamma_0(20),\chi_5) & \text{if } k \in \frac{1}{2}+\mathbb{Z},\\
	\mathcal{M}_k(\varGamma_0(20),\chi_{-4}) & \text{if } k \in 1+2\mathbb{Z},\\
	\mathcal{M}_k(\varGamma_0(20)) & \text{if } k \in 2\mathbb{Z},
	\end{cases}
	\quad
	A_{k,1} \subset 
	\begin{cases}
	\mathcal{M}_k(\varGamma_0(20)) & \text{if } k \in \frac{1}{2}+\mathbb{Z},\\
	\mathcal{M}_k(\varGamma_0(20),\chi_{-20}) & \text{if } k \in 1+2\mathbb{Z},\\
	\mathcal{M}_k(\varGamma_0(20),\chi_{5}) & \text{if } k \in 2\mathbb{Z}.
	\end{cases}
	$$

	Computing the Fourier coefficients, we obtain
	$\operatorname{ord}_{\infty}(f_1) = \operatorname{ord}_{\infty}(f_2) = 0$,
	$\operatorname{ord}_{\infty}(f_3) = 1$,
	$\operatorname{ord}_{\infty}(f_4) = \operatorname{ord}_{\infty}(f_6) = 4$,
	$\operatorname{ord}_{\infty}(f_5) = 5$,
	$\operatorname{ord}_{\infty}(f_7) = 6$.
	Consequently,
	$
	\operatorname{ord}_{\infty}
	(f^{\alpha}) 
	= \alpha_3+4\alpha_4+5\alpha_5+4\alpha_6+6\alpha_7.
	$
	For $ k \in \frac{1}{2}\mathbb{Z} $ with $ 0 < k \leqslant 4 $, 
	we can prove that $ A_{k,0} $ and $ A_{k,1} $ are linearly independent 
	by computing sufficiently many Fourier coefficients.
	For $k \in \frac{1}{2}\mathbb{Z} $ with $ k > 4 $, 
	it is easy to verify that any two modular forms in $ B_{k,0} $ 
	have distinct orders at $ \infty $, 
	and these orders are at most $ 5 $. 
	The same holds for $ B_{k,1} $.
	Moreover, since
	$
	A_{k,0} = f_7 A_{k-2,1} \sqcup B_{k,0},
	A_{k,1} = f_7 A_{k-2,0} \sqcup B_{k,1},
	$
	and the modular forms in $ f_7 A_{k-2,0} $ and $ f_7 A_{k-2,1} $ 
	have orders at $ \infty $ at least $ 6 $, 
	a recursive argument then shows that $ A_{k,0} $ and $ A_{k,1} $ 
	are linearly independent.

	(3) By Lemmas~\ref{lem:1}--\ref{lem:4}, for every half-integer $k \geqslant 0$ we have
	$$
	\dim \mathcal{M}_k^{\mathrm{quad}}(\varGamma_0(20)) = 
		\begin{cases}
		1   & \text{if } k=0,\\
		6k   & \text{if }  k>0,\;2k \equiv 0\pmod 4,\\
		6k-1   & \text{if }  2k \equiv 1\pmod 4,\\
		6k-2   & \text{if }  2k \equiv 2\pmod 4,\\
		6k-3   & \text{if }  2k \equiv 3\pmod 4,
		\end{cases}
	$$
	$$
	\dim \mathcal{S}_k^{\mathrm{quad}}(\varGamma_0(20)) = 
		\begin{cases}
		0   & \text{if } 0 \leqslant k < 2,\\
		1   & \text{if } k = 2,\\
		6k-9   & \text{if }  k>2,\;2k \equiv 1\pmod 4,\\
		6k-10   & \text{if }  k>2,\;2k \equiv 2\pmod 4,\\
		6k-11   & \text{if }  k>2,\;2k \equiv 3\pmod 4,\\
		6k-12   & \text{if }  k>2,\;2k \equiv 0\pmod 4.
		\end{cases}
	$$
\end{proof}

\begin{proof}[Sketch of proof of Theorem~\ref{tho:5}]
	(1) Let the three eta quotients
	$\frac{\eta(2z)^5}{\eta(z)^2\eta(4z)^2}$,
	$\frac{\eta(4z)^5}{\eta(2z)^2\eta(8z)^2}$,
	$\frac{\eta(16z)^2}{\eta(8z)}$,
	$\frac{\eta(32z)^2}{\eta(16z)}$
	be respectively denoted by $f_1$, $f_2$, $f_3$,$f_4$.
	By Lemma~\ref{lem:8},\ref{lem:9},
	we have $f_1\in \mathcal{M}_{1/2}(\varGamma_0(4) ) $,
	$f_2\in \mathcal{M}_{1/2}(\varGamma_0(8),\chi_{8} ) $,
	$f_3\in \mathcal{M}_{1/2}(\varGamma_0(16)) $,
	$f_4\in \mathcal{M}_{1/2}(\varGamma_0(32),\chi_8) $.
	Let $I_{32}=\langle x_1^2-x_2^2+8x_3^2-4x_1x_3,x_3^2+2x_4^2-x_2x_4\rangle$.
	Define the map
	$$
	\bar{\phi}_{32} : \mathbb{C}[x_1,x_2,x_3,x_4]/I_{32} \longrightarrow \mathbb{C}[f_1,f_2,f_3,f_4],\quad
	x_i+I_{32} \mapsto f_i\ (i=1,2,3,4).
	$$
	
	(2) We choose a graded reverse lexicographic order 
	on $\mathbb{C}[x_1,x_2,x_3,x_4]$ 
	with $\deg(x^{\alpha}) = \alpha_1 + \alpha_2 + \alpha_3 + \alpha_4$ 
	and $x_1 > x_2 > x_3 > x_4 $.
	Under this order, 
	$\{ x_{1}^{2} - x_{2}^{2} - 4 x_{1} x_{3} + 8 x_{2} x_{4} - 16 x_{4}^{2}, 
	x_{3}^{2}-x_{2}x_{4}+2x_{4}^{2} \}$
	is a Gr\"{o}bner basis for $I_{32}$.
	Furthermore, 
	we obtain 
	$\operatorname{LT}(I_{32}) = 
	\langle x_1^2, x_3^2\rangle$.
	Let $ S = \{ x^{\alpha} + I_{32} \mid x^{\alpha} 
	\notin \operatorname{LT}(I_{32}) \} $. 

	Given $k\in\frac{1}{2}\mathbb{Z},k>0$,
	let
	$$
		A_{k,0}=
			\{f^{\alpha}
			\mid \alpha_1+\alpha_2+\alpha_3+\alpha_4=2k,
			\alpha_1,\alpha_3 \in \{0,1\},
			\alpha_2,\alpha_4 \in \mathbb{N},\; \alpha_1+\alpha_3 \equiv 0 \!\!\!\!\pmod 2\},
	$$
	$$
		A_{k,1}=
			\{f^{\alpha}
			\mid \alpha_1+\alpha_2+\alpha_3+\alpha_4=2k,
			\alpha_1,\alpha_3 \in \{0,1\},
			\alpha_2,\alpha_4 \in \mathbb{N},\; \alpha_1+\alpha_3 \equiv 1 \!\!\!\!\pmod 2\}.
	$$
	By $\operatorname{LT}(I_{32}) = 
	\langle x_1^2, x_3^2\rangle$,
	we obtain
	$$
		\bar{\phi}_{32}(S)=\{1\}\bigsqcup
		\left(\bigsqcup_{k\in\frac{1}{2}\mathbb{Z},k>0,i\in\{0,1\}} A_{k,i}\right).
	$$
	Moreover, since
	$$
	A_{k,0} \subset 
	\begin{cases}
	\mathcal{M}_k(\varGamma_0(32),\chi_8) & \text{if } k \in \frac{1}{2}+\mathbb{Z},\\
	\mathcal{M}_k(\varGamma_0(32),\chi_{-4}) & \text{if } k \in 1+2\mathbb{Z},\\
	\mathcal{M}_k(\varGamma_0(32)) & \text{if } k \in 2\mathbb{Z},
	\end{cases}
	\quad
	A_{k,1} \subset 
	\begin{cases}
	\mathcal{M}_k(\varGamma_0(32)) & \text{if } k \in \frac{1}{2}+\mathbb{Z},\\
	\mathcal{M}_k(\varGamma_0(32),\chi_{-8}) & \text{if } k \in 1+2\mathbb{Z},\\
	\mathcal{M}_k(\varGamma_0(32),\chi_{8}) & \text{if } k \in 2\mathbb{Z}.
	\end{cases}
	$$

	Computing the Fourier coefficients, we obtain
	$\operatorname{ord}_{\infty}(f_1) = \operatorname{ord}_{\infty}(f_2) = 0$,
	$\operatorname{ord}_{\infty}(f_3) = 1$,
	$\operatorname{ord}_{\infty}(f_4) = 2$.
	Consequently,
	$
	\operatorname{ord}_{\infty}
	(f^{\alpha}) 
	= \alpha_3+2\alpha_4.
	$
	For $ k \in \frac{1}{2}\mathbb{Z} $, $ k > 0 $, 
	it is easy to verify that any two modular forms in $ A_{k,0} $ 
	have distinct orders at $ \infty $. 
	Therefore $ A_{k,0} $ is linearly independent.
	The situation for $ A_{k,1} $ is almost the same as that for $ A_{k,0} $. 

	(3) By Lemmas~\ref{lem:1}--\ref{lem:4}, for every half-integer $k \geqslant 0$ we have
	$$
	\dim \mathcal{M}_k^{\mathrm{quad}}(\varGamma_0(32)) = 
		\begin{cases}
		1   & \text{if } k=0,\\
		8k & \text{if } k>0,
		\end{cases}\quad 
	\dim \mathcal{S}_k^{\mathrm{quad}}(\varGamma_0(32)) = 
		\begin{cases}
		0   & \text{if } 0 \leqslant k < 2,\\
		1   & \text{if } k=2,\\
		8k-16 & \text{if } k>2.
		\end{cases}
	$$
\end{proof}

\section{Proof of Theorem~\ref{tho:6}--\ref{tho:8}}

For a positive integer $k$, 
let 
$
r_k(n) = \#\{(x_1,\dots,x_k) \in \mathbb{Z}^k \mid x_1^2 + \cdots + x_k^2 = n\}.
$
Then we have the following identity:
\[
\sum_{n=0}^{\infty} r_k(n) q^n = \left( \sum_{m=-\infty}^{\infty} q^{m^2} \right)^k
=\frac{\eta(2z)^{5k}}{\eta(z)^{2k}\eta(4z)^{2k}}.
\]

\begin{proof}[Proof of Theorem~\ref{tho:6}]
	By Theorem~\ref{tho:5}, 
	we obtain a basis of $\mathcal{M}_{3/2}(\varGamma_0(32))$:
	\[
		\frac{\eta(2z)^{15}}{\eta(z)^6\eta(4z)^6},
		\frac{\eta(2z)\eta(4z)^8}{\eta(z)^2\eta(8z)^4},
		\frac{\eta(2z)^{10}\eta(16z)^2}{\eta(z)^4\eta(4z)^4\eta(8z)},
	\] 
	\[
		\frac{\eta(4z)^{10}\eta(16z)^2}{\eta(2z)^4\eta(8z)^5},
		\frac{\eta(2z)^3\eta(4z)^3\eta(32z)^2}{\eta(z)^2\eta(8z)^2\eta(16z)},
		\frac{\eta(4z)^5\eta(16z)\eta(32z)^2}{\eta(2z)^2\eta(8z)^3}.
	\]
	The first two eta quotients 
	form a basis of $\mathcal{M}_{3/2}(\varGamma_0(8))$, 
	and the first four eta quotients 
	form a basis of $\mathcal{M}_{3/2}(\varGamma_0(16))$.
	By Lemma~\ref{lem:6} and Lemma~\ref{lem:7}, 
	together with the computation of sufficiently many Fourier coefficients, 
	we obtain the following equalities:
	\begin{align*}
		\frac{\eta(2z)\eta(4z)^8}{\eta(z)^2\eta(8z)^4}
		&= \left. \frac{\eta(2z)^{15}}{\eta(z)^6\eta(4z)^6}\right|S_{4,0}
		+ \frac{1}{3}\left. \frac{\eta(2z)^{15}}{\eta(z)^6\eta(4z)^6}\right|S_{4,1}
		+ \frac{1}{3}\left. \frac{\eta(2z)^{15}}{\eta(z)^6\eta(4z)^6}\right|S_{4,2}
		\\
		&\quad + \left. \frac{\eta(2z)^{15}}{\eta(z)^6\eta(4z)^6}\right|S_{4,3},
		\\
		\frac{\eta(2z)^{10}\eta(16z)^2}{\eta(z)^4\eta(4z)^4\eta(8z)}
		&= \frac{1}{6}\left. \frac{\eta(2z)^{15}}{\eta(z)^6\eta(4z)^6}\right|S_{4,1}
		+ \frac{1}{3}\left. \frac{\eta(2z)^{15}}{\eta(z)^6\eta(4z)^6}\right|S_{4,2}
		+ \frac{1}{2}\left. \frac{\eta(2z)^{15}}{\eta(z)^6\eta(4z)^6}\right|S_{8,3},
		\\
		\frac{\eta(4z)^{10}\eta(16z)^2}{\eta(2z)^4\eta(8z)^5}
		&= \frac{1}{6}\left. \frac{\eta(2z)^{15}}{\eta(z)^6\eta(4z)^6}\right|S_{4,1}
		+ \frac{1}{2}\left. \frac{\eta(2z)^{15}}{\eta(z)^6\eta(4z)^6}\right|S_{8,3},
		\\
		\frac{\eta(2z)^3\eta(4z)^3\eta(32z)^2}{\eta(z)^2\eta(8z)^2\eta(16z)}
		&= \frac{1}{12}\left. \frac{\eta(2z)^{15}}{\eta(z)^6\eta(4z)^6}\right|S_{4,2}
		+ \frac{1}{4}\left. \frac{\eta(2z)^{15}}{\eta(z)^6\eta(4z)^6}\right|S_{8,3}
		+ \frac{1}{3}\left. \frac{\eta(2z)^{15}}{\eta(z)^6\eta(4z)^6}\right|S_{16,4}\\
		&\quad + \frac{1}{6}\left. \frac{\eta(2z)^{15}}{\eta(z)^6\eta(4z)^6}\right|S_{8,5}
		+ \frac{1}{3}\left. \frac{\eta(2z)^{15}}{\eta(z)^6\eta(4z)^6}\right|S_{16,8},
		\\
		\frac{\eta(4z)^5\eta(16z)\eta(32z)^2}{\eta(2z)^2\eta(8z)^3}
		&= \frac{1}{8}\left. \frac{\eta(2z)^{15}}{\eta(z)^6\eta(4z)^6}\right|S_{8,3}
		+ \frac{1}{12}\left. \frac{\eta(2z)^{15}}{\eta(z)^6\eta(4z)^6}\right|S_{8,5}.
	\end{align*}
	Consequently, we obtain the following bases:
	\begin{itemize}
	\item A basis of $\mathcal{M}_{3/2}(\varGamma_0(8))$:
	\[
	\left. \frac{\eta(2z)^{15}}{\eta(z)^6\eta(4z)^6}\right|S_{4,0} + \left. \frac{\eta(2z)^{15}}{\eta(z)^6\eta(4z)^6}\right|S_{4,3},\;
	\left. \frac{\eta(2z)^{15}}{\eta(z)^6\eta(4z)^6}\right|S_{4,1} + \left. \frac{\eta(2z)^{15}}{\eta(z)^6\eta(4z)^6}\right|S_{4,2}.
	\]
	\item A basis of $\mathcal{M}_{3/2}(\varGamma_0(16))$:
	\[
	\left. \frac{\eta(2z)^{15}}{\eta(z)^6\eta(4z)^6}\right|S_{4,0},\;
	\left. \frac{\eta(2z)^{15}}{\eta(z)^6\eta(4z)^6}\right|S_{4,1},\;
	\left. \frac{\eta(2z)^{15}}{\eta(z)^6\eta(4z)^6}\right|S_{4,2},\;
	\left. \frac{\eta(2z)^{15}}{\eta(z)^6\eta(4z)^6}\right|S_{8,3}.
	\]
	\item A basis of $\mathcal{M}_{3/2}(\varGamma_0(32))$:
	\[
	\left. \frac{\eta(2z)^{15}}{\eta(z)^6\eta(4z)^6}\right|S_{16,0} + \left. \frac{\eta(2z)^{15}}{\eta(z)^6\eta(4z)^6}\right|S_{16,12},\;
	\left. \frac{\eta(2z)^{15}}{\eta(z)^6\eta(4z)^6}\right|S_{8,1},\;
	\left. \frac{\eta(2z)^{15}}{\eta(z)^6\eta(4z)^6}\right|S_{4,2},
	\] 
	\[
	\left. \frac{\eta(2z)^{15}}{\eta(z)^6\eta(4z)^6}\right|S_{8,3},\;
	\left. \frac{\eta(2z)^{15}}{\eta(z)^6\eta(4z)^6}\right|S_{16,4} + \left. \frac{\eta(2z)^{15}}{\eta(z)^6\eta(4z)^6}\right|S_{16,8},\;
	\left. \frac{\eta(2z)^{15}}{\eta(z)^6\eta(4z)^6}\right|S_{8,5}.
	\]
	\end{itemize}

	By \cite[Chapter 4, Theorem 1]{Grosswald1985}, 
	$r_3(n) \neq 0$ if and only if 
	$n$ is not of the form $4^a(8b + 7)$ with $a,b \in \mathbb{N}$.
	$
	\frac{\eta(2z)^{15}}{\eta(z)^{6}\eta(4z)^{6}} = \sum_{n=0}^{\infty} r_3(n) q^n.
	$
	Consequently, the Fourier coefficients of
	\[
	\left. \frac{\eta(2z)^{15}}{\eta(z)^6\eta(4z)^6}\right|S_{16,0} + \left. \frac{\eta(2z)^{15}}{\eta(z)^6\eta(4z)^6}\right|S_{16,12}
	\]
	exhibit no periodic sign pattern, whereas the Fourier coefficients of
	\[
	\left. \frac{\eta(2z)^{15}}{\eta(z)^6\eta(4z)^6}\right|S_{8,1},\;
	\left. \frac{\eta(2z)^{15}}{\eta(z)^6\eta(4z)^6}\right|S_{4,2},\;
	\left. \frac{\eta(2z)^{15}}{\eta(z)^6\eta(4z)^6}\right|S_{8,3},\]
	\[
	\left. \frac{\eta(2z)^{15}}{\eta(z)^6\eta(4z)^6}\right|S_{16,4} + \left. \frac{\eta(2z)^{15}}{\eta(z)^6\eta(4z)^6}\right|S_{16,8},\;
	\left. \frac{\eta(2z)^{15}}{\eta(z)^6\eta(4z)^6}\right|S_{8,5}
	\]
	do have periodic sign patterns.
	Thus, 
	when $ a(0) \neq 0 $, 
	the sign $\operatorname{sgn}(a(n))$ is not periodic.
	When $ a(0) = 0 $, 
	the sign $\operatorname{sgn}(a(n))$ is periodic. 
	A simple analysis of the basis of $\mathcal{M}_{3/2}(\varGamma_0(2^r))$ 
	then shows that the period $T$ of $f$ satisfies $T \mid 2^{r-1}$.
\end{proof}

The proofs of Theorem~\ref{tho:7} and Theorem~\ref{tho:8} 
	are similar to that of Theorem~\ref{tho:6}. 
	Therefore we only present the key steps.

\begin{proof}[Sketch of proof of Theorem~\ref{tho:7}]
	By Theorem~\ref{tho:5}, 
	we obtain a basis of $\mathcal{M}_{3/2}(\varGamma_0(32),\chi_8)$:
	\[
    \frac{\eta(2z)^8\eta(4z)}{\eta(z)^4\eta(8z)^2},\;
    \frac{\eta(4z)^{15}}{\eta(2z)^{6}\eta(8z)^6},\;
    \frac{\eta(2z)^3\eta(4z)^3\eta(16z)^2}{\eta(z)^2\eta(8z)^3},
	\]
    \[
	\frac{\eta(2z)^{10}\eta(32z)^2}{\eta(z)^4\eta(4z)^4\eta(16z)},\;
    \frac{\eta(4z)^{10}\eta(32z)^2}{\eta(2z)^4\eta(8z)^4\eta(16z)},\;
    \frac{\eta(2z)^5\eta(16z)\eta(32z)^2}{\eta(z)^2\eta(4z)^2\eta(8z)}.
	\]
	The first two eta quotients 
	form a basis of $\mathcal{M}_{3/2}(\varGamma_0(8),\chi_8)$, 
	and the first three eta quotients 
	form a basis of $\mathcal{M}_{3/2}(\varGamma_0(16),\chi_8)$.
	By Lemma~\ref{lem:6} and Lemma~\ref{lem:7}, 
	together with the computation of sufficiently many Fourier coefficients, 
	we obtain the following equalities:
	\begin{align*}
		\frac{\eta(4z)^{15}}{\eta(2z)^{6}\eta(8z)^6}
    	&=\left. \frac{\eta(2z)^8\eta(4z)}{\eta(z)^4\eta(8z)^2}\right|S_{2,0},
		\\
		\frac{\eta(2z)^3\eta(4z)^3\eta(16z)^2}{\eta(z)^2\eta(8z)^3}
		&=\frac{1}{4}\left. \frac{\eta(2z)^8\eta(4z)}{\eta(z)^4\eta(8z)^2}\right|S_{2,1}
		+\frac{1}{3}\left. \frac{\eta(2z)^8\eta(4z)}{\eta(z)^4\eta(8z)^2}\right|S_{8,2}
		+\frac{1}{3}\left. \frac{\eta(2z)^8\eta(4z)}{\eta(z)^4\eta(8z)^2}\right|S_{8,4},
		\\
		\frac{\eta(2z)^{10}\eta(32z)^2}{\eta(z)^4\eta(4z)^4\eta(16z)}
		&=\frac{1}{6}\left. \frac{\eta(2z)^8\eta(4z)}{\eta(z)^4\eta(8z)^2}\right|S_{8,2}
		+\frac{1}{2}\left. \frac{\eta(2z)^8\eta(4z)}{\eta(z)^4\eta(8z)^2}\right|S_{4,3}
		+\frac{1}{3}\left. \frac{\eta(2z)^8\eta(4z)}{\eta(z)^4\eta(8z)^2}\right|S_{8,4}\\
		&\quad +\frac{1}{2}\left. \frac{\eta(2z)^8\eta(4z)}{\eta(z)^4\eta(8z)^2}\right|S_{16,6},
		\\
		\frac{\eta(4z)^{10}\eta(32z)^2}{\eta(2z)^4\eta(8z)^4\eta(16z)}
		&=\frac{1}{6}\left. \frac{\eta(2z)^8\eta(4z)}{\eta(z)^4\eta(8z)^2}\right|S_{8,2}
		+\frac{1}{3}\left. \frac{\eta(2z)^8\eta(4z)}{\eta(z)^4\eta(8z)^2}\right|S_{8,4}
		+\frac{1}{2}\left. \frac{\eta(2z)^8\eta(4z)}{\eta(z)^4\eta(8z)^2}\right|S_{16,6},
		\\
		\frac{\eta(2z)^5\eta(16z)\eta(32z)^2}{\eta(z)^2\eta(4z)^2\eta(8z)}
		&=\frac{1}{8}\left. \frac{\eta(2z)^8\eta(4z)}{\eta(z)^4\eta(8z)^2}\right|S_{4,3}+\frac{1}{6}\left. \frac{\eta(2z)^8\eta(4z)}{\eta(z)^4\eta(8z)^2}\right|S_{8,4}.
	\end{align*}

	Let
	$
	\frac{\eta(2z)^8\eta(4z)}{\eta(z)^4\eta(8z)^2} = \sum_{n=0}^{\infty} b(n) q^n.
	$
	By Lemma~\ref{lem:6} and Lemma~\ref{lem:7}, 
	we obtain
	\[
	\frac{\eta(2z)^8\eta(4z)}{\eta(z)^4\eta(8z)^2}
	= \frac{1}{3}\left. \frac{\eta(2z)^{15}}{\eta(z)^6\eta(4z)^6}\right|U(2)
	+ \frac{2}{3}\left. \frac{\eta(2z)^{15}}{\eta(z)^6\eta(4z)^6}\right|V(2).
	\]
	Consequently,
	\[
	b(n) = \begin{cases}
	r_3(2n) & \text{if } n \equiv 0 \pmod{2},\\[4pt]
	\dfrac{1}{3}\, r_3(2n) & \text{if } n \equiv 1 \pmod{2}.
	\end{cases}
	\]
	Moreover, 
	by \cite[Chapter 4, Theorem 1]{Grosswald1985}, 
	we have that $b(n) \neq 0$ if and only if $n$ is not of the form $4^c(16d + 14)$ with $c,d \in \mathbb{N}$.
\end{proof}

\begin{proof}[Sketch of proof of Theorem~\ref{tho:8}]
	By Theorem~\ref{tho:3}, 
	we obtain a basis of $\mathcal{M}_{2}(\varGamma_0(16))$:
	\[
		\frac{\eta(2z)^{20}}{\eta(z)^{8}\eta(4z)^{8}},\;
		\frac{\eta(2z)^{6}\eta(4z)^{6}}{\eta(z)^{4}\eta(8z)^{4}},\;
		\frac{\eta(4z)^{20}}{\eta(2z)^{8}\eta(8z)^{8}},\;
		\frac{\eta(2z)^{15}\eta(16z)^{2}}{\eta(z)^{6}\eta(4z)^{6}\eta(8z)},\;
		\frac{\eta(2z)\,\eta(4z)^{8}\eta(16z)^{2}}{\eta(z)^{2}\eta(8z)^{5}}
	\]
	The first three eta quotients 
	form a basis of $\mathcal{M}_{2}(\varGamma_0(8))$.
	Moreover, the eta quotients
	$\frac{\eta(2z)^{20}}{\eta(z)^{8}\eta(4z)^{8}}$,
	$\frac{\eta(4z)^8}{\eta(2z)^4}$
	form a basis of $\mathcal{M}_{2}(\varGamma_0(4))$.
	By Lemma~\ref{lem:5} and Lemma~\ref{lem:7}, 
	together with the computation of sufficiently many Fourier coefficients, 
	we obtain the following equalities:
	\begin{align*}
		\frac{\eta(4z)^8}{\eta(2z)^4}
		&= \frac{1}{8}\left. \frac{\eta(2z)^{20}}{\eta(z)^{8}\eta(4z)^{8}}\right|{S_{2,1}},
		\\
		\frac{\eta(2z)^{6}\eta(4z)^{6}}{\eta(z)^{4}\eta(8z)^{4}}
		&= \left.\frac{\eta(2z)^{20}}{\eta(z)^{8}\eta(4z)^{8}}\right|{S_{4,0}}
		+ \frac{1}{2}\left. \frac{\eta(2z)^{20}}{\eta(z)^{8}\eta(4z)^{8}}\right|{S_{2,1}}
		+ \frac{1}{3}\left.\frac{\eta(2z)^{20}}{\eta(z)^{8}\eta(4z)^{8}}\right|{S_{4,2}}, 
		\\
		\frac{\eta(4z)^{20}}{\eta(2z)^{8}\eta(8z)^{8}}
		&= \left.\frac{\eta(2z)^{20}}{\eta(z)^{8}\eta(4z)^{8}}\right|{S_{4,0}}
		+ \frac{1}{3}\left.\frac{\eta(2z)^{20}}{\eta(z)^{8}\eta(4z)^{8}}\right|{S_{4,2}},
		\\
		\frac{\eta(2z)^{15}\eta(16z)^{2}}{\eta(z)^{6}\eta(4z)^{6}\eta(8z)}
		&= \frac{1}{8}\left.\frac{\eta(2z)^{20}}{\eta(z)^{8}\eta(4z)^{8}}\right|{S_{4,1}}
		+ \frac{1}{4}\left.\frac{\eta(2z)^{20}}{\eta(z)^{8}\eta(4z)^{8}}\right|{S_{4,2}}
		+ \frac{3}{8}\left.\frac{\eta(2z)^{20}}{\eta(z)^{8}\eta(4z)^{8}}\right|{S_{4,3}}
		\\
		&\quad + \frac{1}{3}\left.\frac{\eta(2z)^{20}}{\eta(z)^{8}\eta(4z)^{8}}\right|{S_{8,4}},
		\\
		\frac{\eta(2z)\eta(4z)^{8}\eta(16z)^{2}}{\eta(z)^{2}\eta(8z)^{5}}
		&= \frac{1}{8}\left.\frac{\eta(2z)^{20}}{\eta(z)^{8}\eta(4z)^{8}}\right|{S_{2,1}}
		+ \frac{1}{12}\left.\frac{\eta(2z)^{20}}{\eta(z)^{8}\eta(4z)^{8}}\right|{S_{4,2}}
		+ \frac{1}{3}\left.\frac{\eta(2z)^{20}}{\eta(z)^{8}\eta(4z)^{8}}\right|{S_{8,4}}.
	\end{align*}
	We know
	$
	\frac{\eta(2z)^{20}}{\eta(z)^{8}\eta(4z)^{8}} = \sum_{n=0}^{\infty} r_4(n) q^n.
	$
	By \cite[Chapter 3, Theorem 2]{Grosswald1985},
	$r_4(n) > 0$ for every
	$n \in \mathbb{N}$.
	
\end{proof}

\bibliographystyle{amsalpha}
\bibliography{refs} 

\end{document}